\documentclass{article} 

\usepackage[dvipsnames,svgnames,table]{xcolor}

\usepackage[colorinlistoftodos,textsize=footnotesize,backgroundcolor=yellow!40,bordercolor=red,linecolor=red]{todonotes}

\usepackage{titling}   
\usepackage{titlesec}  

\usepackage{mathtools} 
\usepackage[skins,theorems]{tcolorbox} 
\usepackage{amsmath}   
\usepackage{amsthm}    
\usepackage{amssymb}   

\usepackage[a4paper, margin=3cm]{geometry}

\usepackage{tikz}
\usepackage{tikz-cd}

\usepackage{lineno}

\newtheorem{theorem}{Theorem} 
[section] 
\newtheorem{lemma}[theorem]{Lemma} 
\newtheorem{proposition}[theorem]{Proposition} 
\newtheorem{corollary}[theorem]{Corollary} 

\theoremstyle{definition} 
\newtheorem{definition}[theorem]{Definition} 

\newtheorem{thmletter}{Theorem}          

\newcommand{\C}{\mathbb{C}}    
\newcommand{\N}{\mathbb{N}}    
\newcommand{\R}{\mathbb{R}}    
\newcommand{\Z}{\mathbb{Z}}    

\newcommand{\forwhich}[0]{{ \ ; \ }} 

\newcommand{\al}{\alpha}
\newcommand{\be}{\beta}

\newcommand{\ga}{\gamma}
\newcommand{\De}{\Delta}
\newcommand{\de}{\delta}

\newcommand{\epv}{\varepsilon}

\newcommand{\et}{\eta}

\newcommand{\Si}{\Sigma}

\newcommand{\na}{\nabla} 

\DeclareMathOperator*{\divergence}{div}
\DeclareMathOperator*{\supp}{supp}

\newcommand\norm[1]{\left\Arrowvert#1\right\Arrowvert} 
\newcommand\abs[1]{\left\vert#1\right\vert} 

\usepackage[colorlinks=true, linkcolor=black, urlcolor=black, citecolor=DarkBlue, filecolor=black]{hyperref} 

\numberwithin{equation}{section}

\pretitle{\begin{center}\sffamily\huge}    
\posttitle{\par\end{center}}               
\preauthor{\begin{center}\large\sffamily}   
\postauthor{\par\end{center}}              
\predate{\begin{center}\large\itshape\sffamily} 
\postdate{\par\end{center}}                

\usepackage{tocloft} 

\title{Stability of the Morse Index for the $p$-harmonic Approximation of Harmonic Maps into Homogeneous Spaces}          
\author{Dominik Schlagenhauf \thanks{Department of Mathematics, ETH Zurich, 8092 Z\"urich, Switzerland.}} 
\date{\today}                

\begin{document}
\maketitle 

\begin{abstract}
\noindent
In the joint work of the author with Da Lio and Rivi\`ere \cite{DLRS25} we studied the stability of the Morse index for Sacks-Uhlenbeck sequences into spheres as $p\searrow2$.
These are critical points of the energy
$$E_p(u) \coloneqq \int_\Sigma \left( 1+\abs{\nabla u}^2\right)^{p/2} \ dvol_\Sigma,$$
where $u:\Sigma \rightarrow S^n$ is a map from a closed Riemannian surface $\Sigma$  into a sphere $ S^n$.
In this paper we extend the results found in \cite{DLRS25} to the case of Sacks-Uhlenbeck sequences into homogeneous spaces, by incorporating the strategy introduced in \cite{BR25}.
In the spirit of \cite{DLRS25}, we show in this setting the upper semicontinuity of the Morse index plus nullity and an improved pointwise estimate of the gradient in the neck regions around blow up points.
\end{abstract}

\tableofcontents
\vspace{10mm} 
\newpage  



\section{Introduction}
Let $(\Sigma^2,h)$ be a closed smooth Riemannian surface and let $({\mathcal{N}}^n,g)$ be an  at least $C^2$ $n$-dimensional closed Riemannian manifold, which we assume to be isometrically embedded in Euclidean space $\R^m$.
Harmonic maps are critical points with respect to outer variations of the Dirichlet energy
 \begin{equation}
\label{DE}
E: W^{1,2}(\Sigma;{\mathcal{N}}) \rightarrow\R; \qquad
E(u) \coloneqq \int_\Sigma \abs{\nabla u}^2\ dvol_\Sigma.
\end{equation}
A fundamental question concerns the existence of nontrivial harmonic maps.
The Dirichlet energy is known to be conformally invariant, to possess a non-compact invariance group, and not to satisfy the Palais-Smale condition.
Thus, classical minmax methods are not applicable to \eqref{DE}, and traditional variational theory cannot be directly used.
One of the first existence results was obtained by Eells and Sampson \cite{ES64}.
Under the additional assumption that the target manifold has non‐positive sectional curvature, they proved existence and decay properties of the harmonic map heat flow, which enabled to construct a harmonic map within each homotopy class.
Sacks and Uhlenbeck’s influential results \cite{SaU, SU82} cover the general case:

\begin{thmletter}
[Sacks, Uhlenbeck \cite{SaU}]
\label{THML: SackUhle}	
If $\pi_2(\mathcal N)=0$, then every homotopy class of maps from $\Sigma$ to $\mathcal N$ contains a minimising harmonic map.
If $\pi_2(\mathcal N)\neq0$, then there exists a generating set for $\pi_2(\mathcal N)$ consisting of conformal branched minimal immersions of harmonic spheres which minimise energy and area in their homotopy classes.
\end{thmletter}

To prove Theorem \ref{THML: SackUhle}, Sacks and Uhlenbeck in the foundational work \cite{SaU} introduced the following subcritical relaxations of the Dirichlet Energy
\begin{equation}\label{Ep}
E_p: W^{1,p}(\Sigma;{\mathcal{N}}) \rightarrow\R; \qquad
E_p(u) \coloneqq \int_\Sigma \left( 1+\abs{\nabla u}^2\right)^{p/2} \ dvol_\Sigma,
\end{equation}
where $p>2$.
Since the energy $E_p$ satisfies the Palais--Smale condition and $W^{1,p}(\Sigma)$ embeds into $C^0(\Sigma)$, Sacks and Uhlenbeck constructed a sequence (as $p\searrow 2$) of smooth critical points of \eqref{Ep} within a fixed free homotopy class.
They showed that energy can concentrate at most at finitely many points, where so-called bubbles start to form and, after rescaling, converge to harmonic spheres, while away from these blow-up points the sequence converges to a harmonic map. 
The phenomenon of bubbling was further understood by Parker \cite{Par96}, who discovered the  ``bubble tree''.
In this context of concentration compactnesses, the following three fundamental questions emerge.
\begin{enumerate}
\item \textbf{Energy Identity}:
Is there any loss of energy in the limit?
That is, does the limit of the energy of the sequence equals the energy of the limiting harmonic map plus the energies of the bubbles?
\item \textbf{Necklessness Property} (or {\em $C^0$-no neck property}):
Is the image of the macroscopic limiting harmonic map attached to the images of the bubbles?
More precisely, are the necks connecting the macroscopic and microscopic scales disappearing in the limit?
\item \textbf{Index}:
Is the Morse index preserved in the limit?
In other words, is the number of directions along which the energy decreases preserved in the limiting harmonic map and the bubbles?
\end{enumerate}
In the case of sequences of harmonic maps the energy identity is due to Jost \cite{Jos91} and Parker \cite{Par96}.
It was further generalized to the setting of conformally invariant Lagrangians by Laurain and Rivi\`ere \cite{LR14}, relying on the previous work of Lin and Rivi\`ere \cite{LiRi02}, where the importance of Lorentz space interpolation was fist observed. 
The necklessness property for harmonic maps was first obtained by Parker \cite{Par96}.
Rivi\`ere and Laurain \cite{LR14} showed the $L^{2,1}$-energy quantization (for harmonic maps), which asserts that no $L^{2,1}$-energy is asymptomatically lost in the necks, and thus by the observations made in \cite{LiRi02} implies the necklessness property.
(Here $L^{2,1}$ denotes a Lorentz space.)
See also \cite{MiRi} for a detailed explanation on the relationship between the $L^{2,1}$-energy quantization and the necklessness property.
\medskip

In contrast, the energy identity and necklessness property do no hold in general in the Sacks-Uhlenbeck setting of $p$-harmonic sequences, as a counterexample constructed by Li and Wang \cite{LiWa-2} shows.
Additional informations relating the parameter $p$ to the degenerating conformal structure of the neck regions are required.
Under additional assumptions on the target manifold, several affirmative results have been established:
The energy identity and necklessness property for Sacks-Uhlenbeck sequences to a sphere is due to Li and Zhu \cite{LZ19} and for sequences to a homogeneous space is due to Bayer and Roberts \cite{BR25}.
However, in \cite{BR25}, the computations were not explicitly carried out; instead, the PDE was rewritten and the rest of the proof was referenced to \cite{LZ19}.
Furthermore, Lamm \cite{Lamm} established the energy identity in the setting of min-max critical points while making use of the Struwe's monotonicity trick (see, e.g. \cite{Struwe}).
\medskip

The Morse index of a critical point is the number of independent directions along which the energy decreases and the nullity is the number of independent directions along which the energy is constant.
Our aim is to understand the asymptotic behaviour of the index of a sequence exhibiting bubbling phenomena.
In general, one cannot expect the limit of the Morse index of the sequence elements to equal the Morse indices of the limiting harmonic map plus the the bubbles, as some negative variations may converge to constant variations in the limit.
For this reason, the best one can hope for is the \textit{lower semi-continuity of the Morse index} and the \textit{upper semi-continuity of the Morse index plus nullity} (extended Morse index).
The lower semi-continuity of the Morse index can be shown by classical arguments once the energy identity is established (see Proposition \ref{LSC} and also \cite{KaSt}, \cite{Riv2}).
In contrast to the lower semi-continuity of the Morse index (see, e.g., \cite{CM} for minimal surfaces), the upper semi-continuity is in general significantly more subtle, as it requires a precise control over the sequence of solutions in regions where compactness is lost.
Da Lio, Gianocca and Rivi\`ere  \cite{DGR22} developed a new method to establish the upper semi-continuity of the extended Morse index for conformally invariant variational problems in two dimensions, including the case of harmonic maps.
This new theory has proven to be highly effective in a variety of problems in geometric analysis, including recent developments on biharmonic maps \cite{Mi, MiRi3}, constant mean curvature surfaces \cite{Work}, Ginzburg--Landau energies \cite{DG}, Ricci shrinkers \cite{Yud25}, Willmore surfaces \cite{MiRi2} and Yang--Mills connections \cite{GL, GLR}.
\medskip

In the previous work by the author, in collaboration with Francesca Da Lio and Tristan Rivi\`ere \cite{DLRS25}, the upper semi-continuity of the extended Morse index was shown for Sacks-Uhlenbeck sequences into the $n$-sphere ${S}^n$ as they converge in the bubble tree sense.
This result relied crucially on the high degree of symmetry of the $n$-sphere $S^n$, and in particular on the global conservation laws arising in the Euler--Lagrange equations of \eqref{DE} and \eqref{Ep}, which are consequences of Noether's theorem.
\medskip

In the present paper, we extend the results from \cite{DLRS25} to the setting of an arbitrary closed homogeneous Riemannian target manifold $\mathcal{N}^n$.
(Recall that a homogeneous manifold is one whose group of isometries acts transitively, e.g. spheres, tori and projective spaces)
This broad extension beyond the sphere case is the key new contribution of our work:
\begin{thmletter}
\label{morseindextheoremfordirichletintr}
The extended Morse index is upper semicontinuous along subsequences of Sacks-Uhlenbeck maps into a homogeneous Riemannian manifold.
\end{thmletter}

We outline in the following the main strategy to prove Theorem \ref{morseindextheoremfordirichletintr}.
Building on H\'elein’s foundational ideas \cite{H91}, Bayer and Roberts \cite{BR25} constructed a framework that expresses the Euler--Lagrange equation of \eqref{Ep} as a conservation law.
After rewriting the equation in a div--curl form, we proceed by adapting the strategy  from \cite{DLRS25} and \cite{DGR22}.
We provide an independent proof of the $L^{2,1}$-energy quantization (different from the one in \cite{BR25}) as it is necessary in establishing the refined gradient estimates in the neck regions.
This allows us to prove that the necks are asymptotically not contributing to the negativity of the second variation.
\medskip

The paper is organized as follows.
In Section 2 (Preliminary Definition and Results) we introduce the setting of the problem in full details and explore conservation laws in homogeneous manifolds. 
These notations will be used throughout the paper.
Section 3 is devoted to proving the $L^{2,1}$-energy quantization theorem for Sacks--Uhlenbeck sequences, extending the sphere-case arguments of \cite{DLRS25} to homogeneous manifolds.
(This result was first obtained in \cite{BR25}, using methods from \cite{LZ19}.)
In Section 4 we obtain a pointwise estimate of the gradient in the neck regions, which is an immediate improvement of the $\varepsilon$-regularity in \cite{SU81}.
This shows that asymptotically there is no loss of energy in the necks.
Section 5 establishes the upper semicontinuity of the extended Morse index by combining the neck estimates with a diagonalisation of the Jacobi operator associated to the second variation of the energies.
Theorem \ref{morseindextheoremfordirichletintr} is shown in Theorem \ref{THM: usc morse index sacks-uhlenbeck with one bubble}.
Finally, for the reader’s convenience, the Appendix includes the proof of the lower semicontinuity of the Morse index.
\medskip

\textbf{Acknowledgments.}
The author is sincerely grateful to Francesca Da Lio and Tristan Rivi\`ere for their continuous support and valuable advice.

\section{Preliminary Definition and Results}
\label{SECTION: Preliminary definition and results}

In this section we formally introduce the setting of the problem and the notations for the reminder of the paper.
Let $(\Sigma,h)$ be a smooth closed Riemann surface.
\begin{definition}
[Homogeneous Riemannian Manifold]
A smooth closed homogeneous Riemannian manifold is a smooth closed Riemannian manifold $(\mathcal N^n,g)$ such that its Lie group of isometries $G=\operatorname{Isom}(\mathcal N)$ acts transitive on $\mathcal N$.
(i.e. for all $q_1,q_2\in \mathcal N$ there exists $\phi\in G$ such that $\phi(q_1)=q_2$)
\end{definition}
\noindent
In the following $(\mathcal N,g)$ denotes a homogeneous Riemannian manifold with group of isometries $G=\operatorname{Isom}(\mathcal N)$.
Let us consider some elementary examples:
\begin{itemize}
\item $\mathcal N=S^n$ is a homogeneous Riemannian manifold, where the group of isometries acts by rotations.
\item $\mathcal N=\mathbb T^n= \R^n/\Z^n$ is a homogeneous Riemannian manifold, where the group of isometries acts by translations.
\item $\mathcal N=\mathbb{CP}^n=\C^{n+1}/\sim$ is a homogeneous Riemannian manifold, where $z\sim w$ if and only if $z=\lambda w$ for some $\lambda \in \C$. The group of isometries is given by $G=U(n+1)/U(1)$, where $U(1)=S^1\subset \C.$ (Similar, $\mathcal N=\mathbb{RP}^n=\R^{n+1}/\sim$ is a homogeneous Riemannian manifold.)
\item $\mathcal N=\mathrm{Gr}(k,n)$ the Grassmannian of $k$-planes in $\mathbb{R}^n$ is a homogeneous Riemannian manifold, where the group of isometries acts by rotations.
\end{itemize}
In the following $O(m)\subset \R^{m\times m}$ denotes the subgroup of orthogonal matrices.

\begin{theorem}
[Moore \cite{M76}]
Any homogeneous Riemannian manifold can be isometrically and equivariantly embedded in some Euclidean space. This means that if $(\mathcal N,g)$ is a homogeneous Riemannian manifold with isometry group $G$, then there exists an isometric embedding $\Phi: \mathcal N \rightarrow \R^m$ and an embedding  $\Pi:G \rightarrow O(m)$ such that for any $\psi \in G$ the following diagram commutes
\begin{equation}
\begin{tikzcd}[ampersand replacement=\&]
\mathcal N
  \arrow[r, "{\Phi}"]
  \arrow[d, "\psi"']
\&
\R^m
  \arrow[d, "\Pi(\psi)"]
\\
\mathcal N
  \arrow[r, "\Phi"]
\&
\R^m.
\end{tikzcd}
\end{equation}
\end{theorem}
\textbf{Assumptions \& Notations:}
Henceforward, we will assume that $\mathcal N^n\subset \R^m$ is a submanifold of $\R^m$ and that its group of isometries $G\subset O(m)$ is a subgroup of $O(m)$.
Furthermore, we denote by $\mathfrak g=T_{id}G$ the Lie algebra of $G$ and $L\coloneqq\dim(G)=\dim(\mathfrak g)$.
We recall that as $G\subset O(m)$ we have $\mathfrak g \subset \mathfrak{so}(m)$.
The second fundamental form of the embedding $\mathcal N \hookrightarrow \R^m$ will be denoted by $\mathbb I_q(\cdot,\cdot)$.
\\[3mm]
\noindent
 For $p \geq 2$ we define the $p-$energy as
\begin{equation}
E_p: W^{1,p}(\Sigma;\mathcal N) \rightarrow\R; \qquad
E_p(u) \coloneqq \int_\Sigma \left( 1+\abs{\nabla u}^2\right)^{p/2} \ dvol_\Sigma.
\end{equation}

\begin{definition}
[$p-$Harmonic Map]
We say that a function $u \in W^{1,p}(\Sigma;\mathcal N)$ is a $p$-harmonic map if it is a critical point of $E_p$ with respect to variations in the target. In that case $u$ satisfies the Euler--Lagrange equation
\begin{equation}
\label{EQ: Euler-Lagrange equations in div form for p harm}
-\divergence\left( \left( 1+\abs{\nabla u}^2\right)^{\frac{p}{2} -1} \nabla u \right)= \left( 1+ \abs{\nabla u}^2 \right)^{\frac{p}{2}-1} \mathbb{I}_u\big(\nabla u , \nabla u\big) \in \R^{m},
\end{equation}
or in non-divergence form
\begin{equation}
\label{EQ: Euler-Lagrange equations in non-div form for p harm}
\De u +\left(\frac{p}{2}-1\right)\frac{\langle \na^2 u, \nabla u \rangle \nabla u}{1+\abs{\nabla u}^2}+ \mathbb{I}_u\big(\nabla u , \nabla u\big) = 0 \in \R^{m}.
\end{equation}
\end{definition}

\begin{lemma}
[$\varepsilon$-regularity]
\label{LEMMA: epsilon reg sacks uhlenbeck for p harm}
There exists an $\varepsilon > 0$, a constant $C>0$ and some $p_0>2$ such that for any $p$-harmonic map $u \in W^{1,p}(\Sigma;\mathcal N)$ with $p\in [2,p_0)$ and any geodesic ball $B_r\subset \Sigma$    if
\begin{equation}
\int_{B_r} \abs{\nabla u}^2 dx \le \varepsilon,
\end{equation}
then
\begin{equation}
\norm{\nabla u_k}_{L^\infty(B_{r/2})}
\le \frac{C}{r} \norm{\nabla u_k}_{L^2(B_{r})}.
\end{equation}
\end{lemma}
\noindent
For a proof see in \cite{SU81} Chapter 3, Main Estimate 3.2 and Lemma 3.4.

\subsection{Conservation Laws in Homogeneous Spaces}

We adopt the strategy used in \cite{BR25} to build a frame on $\mathcal N$, which allows one to write the $p$-harmonic map equation as a conservation law.
This goes back to \cite{H91}.    
We start by showing the following lemma.
\begin{lemma}
For any $q \in\mathcal N$ the map
\begin{equation}
\rho_q: \mathfrak g \rightarrow T_q \mathcal N; \qquad
\rho_q(\mathrm A)\coloneqq\mathrm Aq.
\end{equation}
is well-defined and surjective. 
\end{lemma}

\begin{proof}
Let $q\in \mathcal N$ and let $\mathrm A \in \mathfrak g=T_{id}G$.
We want to show that $\mathrm Aq \in T_q\mathcal N$.
There exists some path $Q:(-\epsilon,\epsilon)\rightarrow G$ such that $Q(0)=id$ and $Q^\prime(0)=\mathrm A$.
Define the path $\gamma:(-\epsilon,\epsilon)\rightarrow \mathcal N$ given by $\gamma(t)\coloneqq Q(t)q$.
Then clearly, $\rho_q(\mathrm A)=\mathrm Aq=Q^\prime(0)q=\gamma^\prime(0) \in T_q\mathcal N$ and therefore $\rho_q$ is well-defined. \\
In the following we show that $\rho_q$ is onto.
Let $X \in T_q\mathcal N$.
Then there exists some geodesic $\gamma:(-\epsilon,\epsilon)\rightarrow \mathcal N$ such that $\gamma(0)=q$ and $\gamma^\prime(0)=X$.
We recall that $\mathcal N\cong G/\operatorname{Stab}(q)$, where $\operatorname{Stab}(q)$ denotes the stabilizer of $q$ with respect to the group action of $G$ on $\mathcal N$.
Hence, by the universal property we can lift the path $\gamma$ to a path $Q:(-\epsilon,\epsilon)\rightarrow G$ such that $\gamma(t)=Q(t)q$.
As $\mathrm A\coloneqq Q^\prime(0) \in \mathfrak g$ we have found $\rho_q(\mathrm A)=X$, showing surjectivity of $\rho_q$.
\end{proof}

\begin{lemma}
[Frame]
\label{LEMMA: Frame for the tangent space}
Let $\mathrm A^1, \dots,\mathrm A^L \in \mathfrak g$ be an orthonormal basis of antisymmetric matrices of $\mathfrak g\subset \mathfrak{so}(m)$ with respect to the inner product on $\R^{m\times m}$.
There exist $L$ smooth vector fields $Y^1,\dots,Y^L \in \Gamma(T\mathcal N)$ such that for any point $q \in \mathcal N$ and any tangent vector $X\in T_q\mathcal N$ one has the decomposition 
\begin{equation}
X=\langle \mathrm A^1 q, X \rangle Y^1 + \dots + \langle \mathrm A^L q, X \rangle Y^L.
\end{equation}
\end{lemma}

\begin{proof}
Let $q\in \mathcal N$.
We observe that $\rho_q|_{\ker(\rho_q)^\perp}$ is an isomorphism.
Let $\sigma_q\coloneqq (\rho_q|_{\ker(\rho_q)^\perp})^{-1}$ and let $\sigma_q^*$ be its adjoint.
To $i=1,\dots,L$ we define
\begin{equation}
Y^i \coloneqq \sum_{j=1}^L \langle \sigma_q^*(\mathrm A^j) ,  \sigma_q^*(\mathrm A^i) \rangle\ \mathrm A^j q
= \sum_{j=1}^L \langle \sigma_q^*(\mathrm A^j) ,  \sigma_q^*(\mathrm A^i) \rangle\ \rho_q(\mathrm A^j).
\end{equation}
Let $X\in T_q\mathcal N$. As $\rho_q|_{\ker(\rho_q)^\perp}$ is an isomorphism we can find some $\mathrm A \in \ker(\rho_q)^\perp \subset \mathfrak g$ such that $\rho_q(\mathrm A)=X$. With $\mathrm A=\sigma_q(X)$ write
\begin{equation}
\mathrm A=\sum_{j=1}^L \langle \mathrm A, \mathrm A^j \rangle \mathrm A^j
=\sum_{j=1}^L \langle \sigma_q(X), \mathrm A^j \rangle \mathrm A^j
=\sum_{j=1}^L \langle X, \sigma_q^*(\mathrm A^j) \rangle \mathrm A^j
\end{equation}
and therefore
\begin{equation}
\label{EQ: uinuwneun29838n92niJNIHNdqw}
X=\sum_{j=1}^L \langle X, \sigma_q^*(\mathrm A^j) \rangle\ \rho_q(\mathrm A^j).
\end{equation}
Now using the identity $\rho_q \circ \sigma_q=id_{T_q\mathcal N}$ we express
\begin{equation}
\begin{aligned}
\sigma_q^*(\mathrm A^j) 
&= \rho_q \circ \sigma_q ( \sigma_q^* (\mathrm A^j))
= \rho_q ( \sigma_q \circ \sigma_q^* (\mathrm A^j))
= \rho_q \left( \sum_{i=1}^L \big\langle\sigma_q \circ \sigma_q^* (\mathrm A^j),\mathrm A^i \big\rangle \mathrm A^i \right) \\
&= \rho_q \left( \sum_{i=1}^L \big\langle \sigma_q^* (\mathrm A^j),\sigma_q^*(\mathrm A^i) \big\rangle \mathrm A^i \right)
=  \sum_{i=1}^L \big\langle \sigma_q^* (\mathrm A^j),\sigma_q^*(\mathrm A^i) \big\rangle\ \rho_q (\mathrm A^i)
\end{aligned}
\end{equation}
Going back to \eqref{EQ: uinuwneun29838n92niJNIHNdqw} we have found
\begin{equation}
\begin{aligned}
X&=\sum_{j=1}^L \langle X, \sigma_q^*(\mathrm A^j) \rangle\ \rho_q(\mathrm A^j)
=\sum_{j=1}^L \left\langle X, \sum_{i=1}^L \big\langle \sigma_q^* (\mathrm A^j),\sigma_q^*(\mathrm A^i) \big\rangle\ \rho_q (\mathrm A^i) \right\rangle\ \rho_q(\mathrm A^j) \\
&=\sum_{i,j=1}^L \big\langle \sigma_q^* (\mathrm A^j),\sigma_q^*(\mathrm A^i) \big\rangle\left\langle X, \rho_q (\mathrm A^i) \right\rangle\ \rho_q(\mathrm A^j) \\
&=\sum_{i=1}^L \left\langle X, \rho_q (\mathrm A^i) \right\rangle \sum_{j=1}^L \big\langle \sigma_q^* (\mathrm A^j),\sigma_q^*(\mathrm A^i) \big\rangle\ \rho_q(\mathrm A^j) \\
&=\sum_{i=1}^L \left\langle X, \rho_q (\mathrm A^i) \right\rangle Y^i.
\end{aligned}
\end{equation}

\end{proof}

\begin{lemma}[Conservation Law]
\label{LEMMA: Conservation Law Homo spaces}
Let $u \in W^{1,p}(\Sigma;\mathcal N)$ be a $p$-harmonic map, $p>2$.
Then for any $\mathrm A\in \mathfrak g \subset \mathfrak{so}(m)$ there holds
\begin{equation}
\divergence\left((1+\abs{\nabla u}^2)^{\frac{p}{2}-1} \langle\nabla u, \mathrm Au\rangle \right)=0.
\end{equation}
\end{lemma}

\begin{proof}
As $\mathrm A\in \mathfrak g=T_{id}G$ we can find a path $Q:(-\epsilon,\epsilon)\rightarrow G$ such that $Q(0)=id$ and $Q^\prime(0)=\mathrm A$.
Let $\gamma:(-\epsilon,\epsilon)\rightarrow \mathcal N;$ $\gamma(t)=Q(t)u(x)$.
Then $\mathrm Au(x)=Q^\prime(0)u(x)=\gamma^\prime(0)\in T_u\mathcal N$.
Furthermore, since $u$ is a $p$-harmonic map we have that $\divergence((1+\abs{\nabla u}^2)^{\frac{p}{2}-1} \nabla u ) \in (T_u\mathcal N)^\perp$.
This gives
\begin{equation}
\begin{aligned}
0&=\left\langle \divergence((1+\abs{\nabla u}^2)^{\frac{p}{2}-1} \nabla u ) , \mathrm Au \right\rangle \\
&=\divergence\left( \left\langle (1+\abs{\nabla u}^2)^{\frac{p}{2}-1} \nabla u  , \mathrm Au \right\rangle \right)
-(1+\abs{\nabla u}^2)^{\frac{p}{2}-1} \underbrace{\left\langle  \nabla u , \mathrm A \nabla u \right\rangle}_{=0},
\end{aligned}
\end{equation}
where we used that $\mathrm A$ is anti-symmetric and hence $v^T\mathrm Av=0$, for all $v\in \R^m$.
\end{proof}

In the sphere case $\mathcal N=S^n$ we have that $G=O(n+1)$ and hence $\mathfrak g=\mathfrak{so}(n+1)$.
For any fixed $i,j=1,\dots,n+1$ define the matrix
\begin{equation}
\mathrm A_{\alpha\beta}=\begin{cases}
1, &\text{ if } (\alpha,\beta)=(i,j), \\
-1, &\text{ if } (\alpha,\beta)=(j,i), \\
0, &\text{ else, }
\end{cases}
\end{equation}
Then $\mathrm A \in \mathfrak g=\mathfrak{so}(n+1)$ and hence we recover the conservation law
\begin{equation}
\divergence\left((1+\abs{\nabla u}^2)^{\frac{p}{2}-1}(u \wedge\nabla u)\right)
= \divergence\left((1+\abs{\nabla u}^2)^{\frac{p}{2}-1} \langle\nabla u, \mathrm Au\rangle \right)=0.
\footnote{Here we use the notation $(u \wedge\nabla u)_{i,j}= \nabla u_i u_j - \nabla u_j u_i $.}
\end{equation}

\begin{theorem} [Conservation Law]
\label{THEOREM: Conservation Law p harm map equation}
Let $\mathrm A^1,\dots,\mathrm A^L$ and $Y^1,\dots,Y^L$ be as in Lemma \ref{LEMMA: Frame for the tangent space}.
Let $u \in W^{1,p}(\Sigma;\mathcal N)$ be a $p$-harmonic map, $p>2$.
Then $u$ satisfies the conservation law
\begin{equation}
-\divergence((1+\abs{\nabla u}^2)^{\frac{p}{2}-1} \nabla u )
= \sum_{i=1}^L \nabla^\perp\mathrm  B^i \cdot \nabla \Upsilon^i,
\end{equation}
where $\nabla^\perp\mathrm  B^i \coloneqq -(1+\abs{\nabla u}^2)^{\frac{p}{2}-1} \langle\nabla u, \mathrm A^iu\rangle$ and $\Upsilon^i \coloneqq Y^i\circ u$.
We remark that for some constant $C=C(\mathcal N)>0$ one has the point wise bounds
\begin{equation}
\abs{\nabla\mathrm  B^i} \le C (1+\abs{\nabla u}^2)^{\frac{p}{2}-1} \abs{\nabla u},
\qquad \text{ and } \qquad
\abs{\nabla \Upsilon^i} \le C  \abs{\nabla u}.
\end{equation}
\end{theorem}

\begin{proof}
Considering $X=(1+\abs{\nabla u}^2)^{\frac{p}{2}-1} \nabla u \in T_u\mathcal N$ in Lemma \ref{LEMMA: Frame for the tangent space} we find with Lemma \ref{LEMMA: Conservation Law Homo spaces}
\begin{equation}
\begin{aligned}
&\hspace{-5mm} -\divergence((1+\abs{\nabla u}^2)^{\frac{p}{2}-1} \nabla u ) \\
&= -\divergence\left( \sum_{i=1}^L \langle (1+\abs{\nabla u}^2)^{\frac{p}{2}-1} \nabla u, \mathrm A^iu \rangle Y^i(u) \right) \\
&= \sum_{i=1}^L  -\underbrace{\divergence \Big(\langle (1+\abs{\nabla u}^2)^{\frac{p}{2}-1} \nabla u, \mathrm A^iu \rangle\Big) }_{=0} \ Y^i(u)
- \langle (1+\abs{\nabla u}^2)^{\frac{p}{2}-1} \nabla u, \mathrm A^iu \rangle \cdot \nabla (Y^i(u)) \\
&= \sum_{i=1}^L \nabla^\perp\mathrm  B^i \cdot \nabla \Upsilon^i.
\end{aligned}
\end{equation}
\end{proof}

\subsection{The Second Variation and the Morse Index}

Following the computations carried out in \cite{DLRS25} (for conformally invariant Lagrangians see also \cite{DGR22}) but in the case of a general target $\mathcal N$ we find the following definitions for the second variation and the Morse index.

\begin{definition} [Morse index of $p$-harmonic maps]
\label{DEFINITION: Morse index and Nullityof p-harmonic maps}
Let $u \in W^{1,p}(\Sigma;\mathcal N)$ be a $p$-harmonic map.
Then we introduce the space of variations as
\begin{equation}
\label{EQ: wigunUBNI8193hf23f}
V_u = \Gamma(u^{-1} T\mathcal N)= \left\{ w \in W^{1,2}(\Sigma; \R^{m}) \forwhich w(x) \in T_{u(x)} \mathcal N, \quad \text{for a.e. } x\in\Sigma \right\}.
\end{equation}
The second variation is given by $Q_u:V_u \rightarrow\R$,
\begin{equation}
\begin{aligned}
Q_u(w) 
&\coloneqq p\ (p-2) \int_\Sigma \left( 1+\abs{\nabla u}^2\right)^{p/2-2} \left(\nabla u \cdot \nabla w \right)^2  \ dvol_\Si \\
&\hspace{20mm}+p \int_\Sigma \left( 1+\abs{\nabla u}^2\right)^{p/2-1} \left[\abs{\nabla w}^2 - \mathbb I_u(\nabla u, \nabla u) \cdot \mathbb I_u(w,w) \right] \ dvol_\Si.
\end{aligned}
\end{equation}
The \underline{Morse index} of $u$ relative to the energy $E_p(u)$:  
\begin{equation}
\operatorname{Ind}_{E_p}(u):=\max \left\{\operatorname{dim}(W) ; W\right. \text{ is a sub vector space of } V_u \text{ s.t. } \left.\left.Q_u\right|_{W\setminus\{0\}}<0\right\}
\end{equation}
and the \underline{Nullity} of $u$ to be
\begin{equation}
\operatorname{Null}_{E_p}(u):= \dim\big(\ker Q_u \big).
\end{equation}
\end{definition}

\subsection{Setting of the Problem}
\label{SUBSECTION: Setting of the Problem}

In this section we introduce the setting and the notations used during the remaining of the paper.
We will follow the strategies introduced in \cite{DLRS25} but adapting the Wente structure to accommodate the conservation law we got in Theorem \ref{THEOREM: Conservation Law p harm map equation}.
Let $p_k>2$, $k \in \N$, be a sequence of exponents with 
\begin{equation}
p_k \searrow 2, \qquad \text{ as } k \rightarrow \infty.
\end{equation}
and let  
 $u_k \in W^{1,p_k}(\Sigma;\mathcal N)$ be a sequence of $p_k$-harmonic maps with uniformly bounded energy, i.e. 
\begin{equation}
\label{EQ: Uniform bound on the energy 8912H2egBeZ}
\sup_k E_{p_k}(u_k) = \sup_k \int_\Sigma (1+\abs{\nabla u_k}^2)^{\frac{p_k}{2}} dvol_{\Sigma}<\infty.
\end{equation}
Thanks to a classical result in concentration compactness theory, see for instance \cite{SU81}, we know that the sequence will converge up to subsequences strongly to a harmonic map away from a finite set of blow up points, where bubbles start to form while passing to the limit.
For our purposes it suffices to consider the simplified case of a single blow up point with only one bubble.
In this case we have the following

\begin{definition} [Bubble tree convergence with one bubble]
\label{Definition: Bubble tree convergence of p harm map with one bubble}
We say that the sequence $u_k$ bubble tree converges to a harmonic map and one single bubble if the following happens:
There exist harmonic maps $u_\infty \in W^{1,2}(\Sigma;\mathcal N)$ and $v_\infty \in W^{1,2}(\C;\mathcal N)$, a sequence of radii $(\delta_k)_{k \in \N} \subset \R_{>0}$, a sequence of points $(x_k)_{k \in \N}\subset\Sigma$ and a blow up point $q\in \Sigma$ such that
\begin{equation}
\label{EQ: Cond of bubb conv ji0wr931jJGA}
\begin{aligned}
&\bullet \quad u_k \rightarrow u_\infty, \qquad \text{ in } C^\infty_{loc}(\Sigma \setminus \{q\}), \text{ as } k \rightarrow \infty, \\
&\bullet \quad v_k(z) \coloneqq u_k\left(x_k + \delta_k z \right) \rightarrow v_\infty(z), \qquad \text{ in } C^\infty_{loc}(\C), \text{ as } k \rightarrow \infty, \\
&\bullet \quad \lim _{\eta \searrow 0} \limsup_{k \rightarrow\infty} \sup _{\delta_k / \eta<\rho<2 \rho<\eta} \int_{B_{2 \rho}\left(x_k\right) \backslash B_\rho\left(x_k\right)}\left|\nabla u_k\right|^2 d v o l_\Sigma=0,
\end{aligned}
\end{equation}
where in the second line $u_k(\cdot)$ is to be understood on a fixed conformal chart around the point $q$ and also
\begin{equation}
\label{EQ: Bubb conv 2}
 x_k \rightarrow q, \quad \de_k \rightarrow 0, \qquad \text{ as } k \rightarrow \infty.
\end{equation}
\end{definition}

\noindent

Henceforward, we will assume that we are in the setting of Definition \ref{Definition: Bubble tree convergence of p harm map with one bubble}.
Furthermore, we are working in a fixed conformal chart around the point $q$ centered at the origin and parametrized by the unit ball $B_1=B_1(0)$.
Also for the sake of simplicity $x_k=0=q$ for any $k\in\N$. \\
We consider the vector field 
\begin{equation}\label{Xk}
X_{k}=(1+\abs{\nabla u_k}^2)^{\frac{p_k}{2}-1} \nabla u_k \in L^{p_k^\prime}(B_1),~~p_k^\prime=\frac{p_k}{p_k-1},
\end{equation}
which satisfies by \eqref{EQ: Euler-Lagrange equations in div form for p harm} the equation
\begin{equation}\label{eqXk}
-\divergence(X_k)= (1+\abs{\nabla u_k}^2)^{\frac{p_k}{2}-1}\  \mathbb{I}_{u_k}\big(\nabla u_k , \nabla u_k\big) \qquad \text{ in } B_1.
\end{equation}
Let $\mathrm A^1,\dots,\mathrm A^L$ be an orthonormal basis of $\mathfrak g$ (with respect to the inner product in $\R^{m\times m}$) and let $Y^1,\dots,Y^L\in \Gamma (T\mathcal N)$ be the smooth vector fields constructed in Lemma \ref{LEMMA: Frame for the tangent space}.
Applying Theorem \ref{THEOREM: Conservation Law p harm map equation} for $k \in \N$ we find 
\begin{equation}
-\divergence(X_k)
= \sum_{i=1}^L \nabla^\perp\mathrm  B_{\eta,k}^i \cdot \nabla \Upsilon_{\eta,k}^i,
\end{equation}
where
\begin{equation}
\label{EQ: Equation for nabla B def}
\nabla^\perp\mathrm  B_{\eta,k}^i = -(1+\abs{\nabla u_k}^2)^{\frac{p_k}{2}-1} \langle\nabla u_k, \mathrm A^iu_k\rangle,
\qquad \text{ and } \qquad
\Upsilon_{\eta,k}^i=Y^i\circ u_k,
\end{equation}
with $\eta>0$.
(Here we are using the subscript $\eta$ for consistency of notation, although non of the quantities has any dependance on it.)
We remark that 
For some constant $C=C(\mathcal N)>0$ (depending only on the embedding of $\mathcal N$) one has the point wise bounds
\begin{equation}
\label{EQ: inwfo9n2038njJNUDW223d}
\abs{\nabla\mathrm B_{\eta,k}^i} \le C (1+\abs{\nabla u_k}^2)^{\frac{p_k}{2}-1} \abs{\nabla u_k},
\qquad \text{ and } \qquad
\abs{\nabla \Upsilon_{\eta,k}^i} \le C  \abs{\nabla u_k}.
\end{equation}
Given $\eta\in (0,1)$, and $k \in \N$, we consider the annulus 
\begin{equation}
A(\eta, \delta_k)\coloneqq B_{\eta}(0)\setminus \overline{B}_{\delta_k/\eta}(0),
\end{equation}
which is called neck-region.
Combining H\"older, \eqref{EQ: Uniform bound on the energy 8912H2egBeZ} and \eqref{EQ: inwfo9n2038njJNUDW223d} we can bound
\begin{equation}
\label{EQ: estimate of B 980jt4unJUN}
\begin{aligned}
\norm{\nabla \mathrm B_{\eta,k}^i}_{L^{p_k^\prime}(A(\eta,\delta_k))} 
\le C \norm{\nabla u_k}_{L^{p_k}(A(\eta,\delta_k))}
, \qquad
\norm{\nabla \mathrm \Upsilon_{\eta,k}^i}_{L^{p_k}(A(\eta,\delta_k))} 
\le C \norm{\nabla u_k}_{L^{p_k}(A(\eta,\delta_k))}.
\end{aligned}
\end{equation}
We use the Hodge/Helmholtz-Weyl Decomposition from Lemma A.6 in \cite{DLRS25} on the domain $\Omega=B_1$ to find some $a,b \in W^{1,p_k^\prime}(B_1)$ such that
\begin{equation}\label{decXk}
X_k= \nabla a_{\eta,k} + \nabla^{\perp}b_{\eta,k}\quad \text{ in } B_1
\end{equation}
and with $\partial_{\tau} b_{\eta,k}=0$ on $\partial B_1$.
We get the equation
\begin{equation}
-\Delta a_{\eta,k}=-\divergence(X_k) = \sum_{i=1}^L \nabla^\perp\mathrm  B_{\eta,k}^i \cdot \nabla \Upsilon_{\eta,k}^i \qquad \text{ in } B_1.
\end{equation}
Let $\widetilde u_k$ be the Whitney extension to $\C$ of $u_k|_{A(\eta,\delta_k)}$ coming from Lemma A.1 of \cite{DLRS25} with
\begin{equation}
\label{EQ: jnfemorjsfJMNi29d91}
\begin{aligned}
\norm{\nabla \widetilde u_k}_{L^{p_k}(\C)} 
\le C\norm{\nabla  u_k}_{L^{p_k}(A(\eta,\delta_k))}
\end{aligned}
\end{equation}
and also
\begin{equation}
\label{EQ: wuinuifnvuUNIDWU0139jr1}
\supp (\nabla \widetilde u_k) \subset A(2\eta,\delta_k).
\end{equation}
Letting $\widetilde \Upsilon_{\eta,k}^i\coloneqq Y_i\circ \widetilde u_k$ we also find
\begin{equation}
\label{EQ: iunfr9e902309niJNI93sJM2jms}
\begin{aligned}
\abs{\nabla \widetilde \Upsilon_{\eta,k}^i} \le C \abs{\nabla \widetilde u_k}
\end{aligned}
\end{equation}
and also
\begin{equation}
\supp (\nabla \widetilde \Upsilon_{\eta,k}^i) \subset A(2\eta,\delta_k).
\end{equation}
Let $\mathrm{\widetilde B}_{\eta,k}^i$ be the Whitney extensions to $\C$ of $\mathrm B_{\eta,k}^i|_{A(\eta,\delta_k)}$ coming from Lemma A.1 of \cite{DLRS25} with
\begin{equation}
\label{EQ: Control of tilde B uhni893hre}
\norm{\nabla \mathrm{\widetilde B}_{\eta,k}^i}_{L^{p_k^\prime}(\C)} 
\le C\norm{\nabla  \mathrm B_{\eta,k}^i}_{L^{p_k^\prime}(A(\eta,\delta_k))}
\end{equation}
and also
\begin{equation}
\supp (\nabla \mathrm{\widetilde B}_{\eta,k}^i) \subset A(2\eta,\delta_k).
\end{equation}
For $i=1,\dots,L$ let $\varphi_{\eta,k}^i\in W^{1,2}(\C)$ be the solution of
\begin{equation}
\label{EQ: Eq for De phi}
-\Delta \varphi_{\eta,k}^i = \nabla^\perp\mathrm{\widetilde B}_{\eta,k}^i \cdot \nabla \widetilde \Upsilon_{\eta,k}^i \qquad \text{ in } \C.
\end{equation}
Letting $\varphi_{\eta,k}\coloneqq \sum_{i=1}^L \varphi_{\eta,k}^i$ we have
\begin{equation}
\label{EQ: Eq for De phi sumed}
-\Delta \varphi_{\eta,k} = \sum_{i=1}^L \nabla^\perp\mathrm{\widetilde B}_{\eta,k}^i \cdot \nabla \widetilde \Upsilon_{\eta,k}^i \qquad \text{ in } \C.
\end{equation}
Now set
\begin{equation}
\mathfrak h_{\eta,k}=a_{\eta,k} - \varphi_{\eta,k} \qquad \text{ in } B_1.
\end{equation}
Clearly, $\mathfrak h_{\eta,k}$ is harmonic in $A(\eta,\delta_k)$.
Now we decompose the harmonic part $\mathfrak h_{\eta,k}$ as follows:
\begin{equation}
\mathfrak h_{\eta,k} =\mathfrak h_{\eta,k}^+ + \mathfrak h_{\eta,k}^- + \mathfrak h_{\eta,k}^0 \qquad \text{ in } A(\eta,\delta_k),
\end{equation}
where
\begin{equation}
\mathfrak h_{\eta,k}^+ = \Re \left(\sum_{l>0} h_l^k z^l \right), \qquad
\mathfrak h_{\eta,k}^- = \Re \left(\sum_{l<0} h_l^k z^l \right), \qquad
\mathfrak h^0_{\eta,k} = h^k_0 + C^k_0 \log\abs{z}.
\end{equation}
From \eqref{decXk} we get the decomposition
\begin{equation}
\label{EQ: Final decomp of na u}
X_k= \nabla^\perp b_{\eta,k}+\nabla \varphi_{\eta,k}+ \nabla \mathfrak h_{\eta,k}^+ + \nabla \mathfrak h_{\eta,k}^- \qquad \text{ in } A(\eta,\delta_k).
\end{equation}

\begin{lemma}
There holds $C^k_0=0$ and hence $\nabla \mathfrak h_{\eta,k}^0=0$.
\end{lemma}

\begin{proof}
Let $r \in \left(\frac{\delta_k}{\eta},\eta \right)$.
Then
\begin{equation}
\label{EQ: 2r381brufr1}
\int_{B_r} \De \mathfrak h_{\eta,k} \ dz
= \int_{B_r} \divergence \nabla \mathfrak h_{\eta,k} \ dz
= \int_{\partial B_r} \partial_\nu \mathfrak h_{\eta,k}^+ \ d\sigma
+ \int_{\partial B_r} \partial_\nu \mathfrak h_{\eta,k}^- \ d\sigma
+ \int_{\partial B_r} \partial_\nu \mathfrak h_{\eta,k}^0 \ d\sigma
\end{equation}
Now we compute
\begin{equation}
\label{EQ: 8ru81bf81b}
\int_{\partial B_r} \partial_\nu \mathfrak h_{\eta,k}^+ \ d\sigma = 0 =
\int_{\partial B_r} \partial_\nu \mathfrak h_{\eta,k}^- \ d\sigma.
\end{equation}
Furthermore, 
\begin{equation}
\label{EQ: Comp flux of h0 edjn293}
\int_{\partial B_r} \partial_\nu \mathfrak h_{\eta,k}^0 \ d\sigma =C^k_0 \int_{\partial B_r} \frac{1}{r} \ d\sigma = 2\pi C^k_0.
\end{equation}
Combining \eqref{EQ: 2r381brufr1}, \eqref{EQ: 8ru81bf81b} and \eqref{EQ: Comp flux of h0 edjn293} we find
\begin{equation}\label{EQ: unifw81389BN92IJ4127dudhAWQKGX}
\begin{aligned}
C^k_0 
= \frac{1}{2\pi} \int_{B_r} \De \mathfrak h_{\eta,k} \ dz = \frac{1}{2\pi} \int_{B_r} \De a_{\eta,k} - \De \varphi_{\eta,k} \ dz
\end{aligned}
\end{equation}
Now we compute
\begin{equation}
\begin{aligned}
\int_{B_r} \De a_{\eta,k} \ dz
&= \sum_{i=1}^L \int_{B_r} \nabla^\perp\mathrm  B_{\eta,k}^i \cdot \nabla \Upsilon_{\eta,k}^i \ dz \\
&= \sum_{i=1}^L \int_{B_r} \divergence\left( \nabla^\perp\mathrm  B_{\eta,k}^i  \Upsilon_{\eta,k}^i\right) \ dz \\
&= \sum_{i=1}^L \int_{\partial B_r} \left( \nabla^\perp\mathrm  B_{\eta,k}^i  \Upsilon_{\eta,k}^i\right) \cdot \nu \ d\sigma \\
&= \sum_{i=1}^L \int_{\partial B_r} \left(\nabla^\perp\mathrm{\widetilde B}_{\eta,k}^i \widetilde \Upsilon_{\eta,k}^i\right) \cdot \nu \ d\sigma \\
&= \sum_{i=1}^L \int_{B_r} \divergence\left( \nabla^\perp\mathrm{\widetilde B}_{\eta,k}^i \widetilde \Upsilon_{\eta,k}^i \right) \ dz \\
&= \sum_{i=1}^L \int_{B_r} \nabla^\perp\mathrm{ \widetilde B}_{\eta,k}^i \cdot \nabla \widetilde\Upsilon_{\eta,k}^i \ dz
= \int_{B_r} \De \varphi_{\eta,k} \ dz
\end{aligned}
\end{equation}
Going back to \eqref{EQ: unifw81389BN92IJ4127dudhAWQKGX} the claim follows.
\end{proof}

\vspace{10mm}
\section{Energy quantization}

In this section we adapt the proof of the $L^{2,1}$-energy quantization from \cite{DLRS25} (see also \cite{DGR22}) to the case of a homogeneous manifold in the target.
The $L^{2,1}$-energy quantization for Sacks-Uhlenbeck sequences in the sphere case is due to \cite{LZ19} and was extended to the setting of homogeneous manifolds in \cite{BR25}.
They were using a different method, which involves a direct cut-off argument on the boundaries of the necks and the application of Wente's inequality.
Our method involves the Whitney type extensions introduced in Section \ref{SUBSECTION: Setting of the Problem} and weighted Wente type inequalities.
The $L^{2,1}$-energy quantization derived in this section is used to obtain the pointwise bound of the gradient in the neck regions in Section \ref{SECTION: Pointwise Control of the Gradient in the Neck Regions}.\medskip

For arguments that are the same as in the sphere case and are rather standard in the literature, we will refer to \cite{DLRS25} and omit carrying out the proof. 
\medskip

The $L^{2,\infty}$-energy quatization is a direct consequence of $\varepsilon$-regularity Lemma \ref{LEMMA: epsilon reg sacks uhlenbeck for p harm}:

\begin{lemma}[$L^{2,\infty}$-energy quantization]\label{THM: L2infty en quant of na uk}
There holds
\begin{equation}
\lim_{\eta\searrow0}\limsup_{k\rightarrow\infty}\norm{\nabla u_k}_{L^{2,\infty}(A(\eta,\delta_k))}=0.
\end{equation}
\end{lemma}

\begin{proof}
By $\varepsilon$-regularity Lemma \ref{LEMMA: epsilon reg sacks uhlenbeck for p harm} it is clear that $\abs{\nabla u_k(x)} \le C \abs{x}^{-1}\norm{\nabla u_k}_{L^2(B_{\abs{x}/4}(x))}$. One concludes using $\abs{x}^{-1} \in L^{2,\infty}$ and \eqref{EQ: Cond of bubb conv ji0wr931jJGA}.
For more details see Theorem 3.2 in \cite{DLRS25}.
\end{proof}

Recall the decomposition constructed in \eqref{EQ: Final decomp of na u}.

\begin{lemma}\label{LEMMA: L21 of na h}
There holds
\begin{equation}
\lim_{\eta\searrow0}\limsup_{k\rightarrow\infty} \norm{\nabla \mathfrak h_{\eta,k}^\pm}_{L^{2,1}(A(\eta,\delta_k))}=0.
\end{equation}
\end{lemma}

\begin{proof}
The proof is the same as   in Lemma III.3 of \cite{DGR22} and we omit it.
\end{proof}

\begin{lemma}
\label{LEMMA: First est of nabla b in necks}
For $k\in \N$ large and $\eta>0$  small one has
\begin{equation}
\label{EQ: HIninfwjei921ejir2}
\norm{\nabla b_{\eta,k}}_{L^{p_k^\prime}(B_1)} \le C (p_k-2)
\end{equation}
\end{lemma}

\begin{proof}
This proof is the same as in Lemma 3.4 of \cite{DLRS25}.
For $p>2$ we consider the operator
\begin{equation}
S_p(f)\coloneqq \left[\frac{1+|f|^2}{1+\|f\|_{L^p(B_1)}^2}\right]^{\frac{p}{2}-1} f
\end{equation}
and let $T(f)=\nabla^\perp B$, where $f=\nabla A + \nabla^\perp B$ and $\partial_\tau B=0$ is the Hodge/Helmholtz-Weyl Decomposition of $f$ as e.g. in Lemma A.6 in \cite{DLRS25}.
Then we can apply Coifman-Rochberg-Weiss commutator type Lemma A.5 of \cite{BR19} and use \eqref{EQ: Uniform bound on the energy 8912H2egBeZ} to derive
\begin{equation}
\norm{\nabla b_{\eta,k}}_{L^{p_k^\prime}(B_1)}
\le C \norm{\left[T,S_{p_k}\right](\nabla u_k)}_{{L^{p_k^\prime}(B_1)}}
\le C (p_k-2).
\end{equation}
For the full intermediate computations see Lemma 3.4 of \cite{DLRS25}.

\end{proof}

\begin{lemma} \label{LEMMA: L21 of na b}
For $\eta>0$ we have
\begin{equation}
\label{EQ: L21 est of nabla b eta k}
\lim_{k \rightarrow\infty} \norm{\abs{\nabla b_{\eta,k}}^{1/(p_k-1)}}_{L^{2,1}(B_1)}=0.
\end{equation}
\end{lemma}

\begin{proof}
This result follows by using H\"older's inequality, computing there the exact constant and using Lemma \ref{LEMMA: First est of nabla b in necks}.
For all the details see Lemma 3.5 in \cite{DLRS25}.
\end{proof}

We can finally show,

\begin{theorem}[$L^2$-energy quantization]
\label{THM: L2 en quant of na u}
There holds
\begin{equation}
\lim_{\eta\searrow0}\limsup_{k\rightarrow\infty}\norm{\nabla u_k}_{L^{2}(A(\eta,\delta_k))}=0.
\end{equation}
\end{theorem}

\begin{proof}
We follow closley the proof of Theorem 3.5 in \cite{DLRS25} but adapt it to accomodate the conservation law coming from Theorem \ref{THEOREM: Conservation Law p harm map equation}.
One can estimate
\begin{equation}
\begin{aligned}
\norm{\abs{X_k}^{\frac{1}{p_k-1}}}_{L^{2,\infty}(A(\eta,\delta_k))}
\le \norm{(1+\abs{\nabla u_k})}_{L^{2,\infty}(A(\eta,\delta_k))} 
\le C \left( \eta + \norm{\nabla u_k}_{L^{2,\infty}(A(\eta,\delta_k))} \right)
\end{aligned}
\end{equation}
and thus with Theorem \ref{THM: L2infty en quant of na uk}
\begin{equation}
\label{EQ: L21 of Xk ifn3N51nM}
\lim_{\eta\searrow0}\limsup_{k\rightarrow\infty}\norm{\abs{X_k}^{\frac{1}{p_k-1}}}_{L^{2,\infty}(A(\eta,\delta_k))}=0.
\end{equation}
Following the computation as in the proof of Theorem 3.5 in \cite{DLRS25} one has for any function $f$ on a bounded domain $\Omega$ and any $p>2$
\begin{equation}
\label{EQ: ZBU3rZfi8U9djogqw}
\begin{aligned}
\norm{f^{\frac{1}{p-1}}}_{L^{2,1}(\Omega)}
\le 2\ \abs{\Omega}^{\frac{1}{2}} + \frac{2}{p-1} \norm{f}_{L^{2,1}(\Omega)}.
\end{aligned}
\end{equation}
Combining \eqref{EQ: ZBU3rZfi8U9djogqw} with Lemma \ref{LEMMA: L21 of na h} we obtain
\begin{equation}
\label{EQ: L21 skewed na h ZH82bn9Ns1}
\lim_{\eta\searrow0}\limsup_{k\rightarrow\infty} \norm{\abs{\nabla \mathfrak h_{\eta,k}^\pm}^{\frac{1}{p_k-1}}}_{L^{2,1}(A(\eta,\delta_k))}=0.
\end{equation}
Going back to the decomposition \eqref{EQ: Final decomp of na u} and using Lemma \ref{LEMMA: L21 of na b}, \eqref{EQ: L21 of Xk ifn3N51nM}, \eqref{EQ: L21 skewed na h ZH82bn9Ns1}, we find
\begin{equation}
\label{EQ: NUunf81b1dvwdJ12J}
\lim_{\eta\searrow0}\limsup_{k\rightarrow\infty} \norm{\abs{\nabla \varphi_{\eta,k}}^{\frac{1}{p_k-1}}}_{L^{2,\infty}(A(\eta,\delta_k))}=0.
\end{equation}
Using Wente's inequality with \eqref{EQ: Eq for De phi} and also \eqref{EQ: jnfemorjsfJMNi29d91}, \eqref{EQ: estimate of B 980jt4unJUN}, \eqref{EQ: Control of tilde B uhni893hre}, \eqref{EQ: Uniform bound on the energy 8912H2egBeZ} one finds
\begin{equation}
\label{EQ: uZG813g87UH}
\norm{\nabla \varphi_{\eta,k}}_{L^{2,1}(\C)} 
\le C \sum_{i=1}^L \norm{\nabla\widetilde{\mathrm B}_{\eta,k}^i}_{L^{p_k^\prime}(\C)} \norm{\nabla\widetilde \Upsilon_{\eta,k}^i}_{L^{p_k}(\C)} 
\le C \norm{\nabla u_{k}}_{L^{p_k}(A(\eta,\delta_k))}^2
\le C.
\end{equation}
Combining \eqref{EQ: uZG813g87UH} with \eqref{EQ: ZBU3rZfi8U9djogqw}
we find
\begin{equation}
\label{EQ: Ifw8193rH23nN}
\norm{\abs{\nabla \varphi_{\eta,k}}^{\frac{1}{p_k-1}}}_{L^{2,1}(A(\eta,\delta_k))} \le C.
\end{equation}
 By H\"older's inequality in Lorentz spaces
\begin{equation}
\norm{\abs{\nabla \varphi_{\eta,k}}^{\frac{1}{p_k-1}}}_{L^{2}(A(\eta,\delta_k))}
\le \norm{\abs{\nabla \varphi_{\eta,k}}^{\frac{1}{p_k-1}}}_{L^{2, \infty}(A(\eta,\delta_k))} \norm{\abs{\nabla \varphi_{\eta,k}}^{\frac{1}{p_k-1}}}_{L^{2, 1}(A(\eta,\delta_k))}
\end{equation}
Hence, using \eqref{EQ: Ifw8193rH23nN} and \eqref{EQ: NUunf81b1dvwdJ12J} we get
\begin{equation}
\label{EQ: UNuwf83hn1H12fd}
\lim_{\eta\searrow0}\limsup_{k\rightarrow\infty} \norm{\abs{\nabla \varphi_{\eta,k}}^{\frac{1}{p_k-1}}}_{L^{2}(A(\eta,\delta_k))}=0.
\end{equation}
Going back to the decomposition \eqref{EQ: Final decomp of na u} and using Lemma \ref{LEMMA: L21 of na b}, \eqref{EQ: L21 skewed na h ZH82bn9Ns1} and \eqref{EQ: UNuwf83hn1H12fd} we obtain
\begin{equation}
\lim_{\eta\searrow0}\limsup_{k\rightarrow\infty}\norm{\abs{X_{k}}^{\frac{1}{p_k-1}}}_{L^{2}(A(\eta,\delta_k))}=0.
\end{equation}
The bound $\abs{\nabla u_k}\le \abs{X_k}^{\frac{1}{p_k-1}}$ gives the claimed result.
\end{proof}

\begin{lemma} \label{LEMMA: Linfty bd on na u in the necks unif}
There is a constant $C>0$ such that for $k\in\N$ large there holds
\begin{equation}\label{linftyuk}
\norm{\left(1+\abs{\nabla u_k}^2 \right)^{\frac{p_k}{2}-1}}_{L^\infty(\Sigma)} \le C.
\end{equation}
\end{lemma}

\begin{proof}
This proof is rather standard in the Sacks-Uhlenbeck bubbling analysis.
One bounds $\|\nabla u_k\|_{L^\infty(\Sigma)} \le C \delta_k^{-1}$ and the result follows by a rescaling argument.
For all the details see \cite{DLRS25} and also \cite{LZ19}.
\end{proof}

\begin{theorem}[$L^{2,1}$-energy quantization]
There holds
\begin{equation}
\lim_{\eta\searrow0}\limsup_{k\rightarrow\infty}\norm{\nabla u_k}_{L^{2,1}(A(\eta,\delta_k))}=0.
\end{equation}
\end{theorem}

\begin{proof}
By following \eqref{EQ: uZG813g87UH} and using Lemma \ref{LEMMA: Linfty bd on na u in the necks unif} one gets
\begin{equation}
\begin{aligned}
\norm{\nabla \varphi_{\eta,k}}_{L^{2,1}(\C)}
\le C \norm{\nabla u_{k}}_{L^{p_k}(A(\eta,\delta_k))}^2
\le C \norm{\nabla u_k}_{L^2(A(\eta,\delta_k))}^{\frac{4}{p_k}}
\end{aligned}
\end{equation}
Using Theorem \ref{THM: L2 en quant of na u} we find
\begin{equation}
\label{EQ: jNU8uj893NIJ1}
\lim_{\eta\searrow0}\limsup_{k\rightarrow\infty} \norm{\nabla \varphi_{\eta,k}}_{L^{2,1}(\C)}=0.
\end{equation}
Using \eqref{EQ: ZBU3rZfi8U9djogqw} we find
\begin{equation}
\label{EQ: un23J2jd10jN012}
\lim_{\eta\searrow0}\limsup_{k\rightarrow\infty} \norm{\abs{\nabla \varphi_{\eta,k}}^{\frac{1}{p_k-1}}}_{L^{2,1}(A(\eta,\delta_k))}=0.
\end{equation}
Going back to the decomposition \eqref{EQ: Final decomp of na u} and combining Lemma \ref{LEMMA: L21 of na b}, \eqref{EQ: L21 skewed na h ZH82bn9Ns1} and \eqref{EQ: un23J2jd10jN012}
\begin{equation}
\lim_{\eta\searrow0}\limsup_{k\rightarrow\infty} \norm{\abs{X_k}^{\frac{1}{p_k-1}}}_{L^{2,1}(A(\eta,\delta_k))}=0.
\end{equation}
The bound $\abs{\nabla u_k}\le \abs{X_k}^{\frac{1}{p_k-1}}$ gives the claimed result.
\end{proof}

\vspace{10mm}
\section{Pointwise Control of the Gradient in the Neck Regions}
\label{SECTION: Pointwise Control of the Gradient in the Neck Regions}

In this section we show an improved pointwise control in the neck regions compared to the control coming from $\varepsilon$-regularity Lemma \ref{LEMMA: epsilon reg sacks uhlenbeck for p harm}:

\begin{theorem}
\label{THEOREM: on Pointwise Estimate of the Gradient in the Necks p harm}
For any given $\be \in \left(0,\log_2 (3/2) \right)$ we find that for $k \in \N$ large and $\eta>0$ small
\begin{equation}
\label{EQ: 283hfunv0lpdnu13ef6zh1}
\begin{gathered}
\forall x \in A(\et,\de_k): \qquad
\abs{x}^2 \abs{\nabla u_k(x)}^2 \le \left[ \left( \frac{\abs{x}}{\et} \right)^\be + \left( \frac{\de_k}{\et \abs{x}} \right)^\be \right]\boldsymbol{\epsilon}_{\eta, \delta_k} +  \boldsymbol{\mathrm c}_{\eta, \delta_k},
\end{gathered}
\end{equation}
where
\begin{equation}
\lim_{\eta\searrow0}\limsup_{k\rightarrow\infty}\boldsymbol{\epsilon}_{\eta, \delta_k}=0,
\qquad \text{and} \qquad
\lim_{\eta\searrow0}\limsup_{k\rightarrow\infty}\boldsymbol{\mathrm c}_{\eta, \delta_k} \log^2\left(\frac{\eta^2}{\delta_k}\right) =0.
\end{equation}
\end{theorem}

We will closley follow the proof of Theorem 4.1 in \cite{DLRS25} and adapt it from the sphere to the homogeneous case.
See also \cite{DGR22}.
Introduce the notation
\begin{equation}
A_j=B_{2^{-j}}\setminus B_{2^{-j-1}} , \qquad j \in \N.
\end{equation}
Now recall that we are working with the decomposition introduced in \eqref{EQ: Final decomp of na u}. \par

\begin{lemma}[Estimate of $\nabla \varphi_{\eta,k}$]\label{LEMMA: Dyadic est of na phi start}
For any $\gamma \in \left(0,\frac{2}{3}\right]$ there is a constant $C=C(\gamma)>0$ such that for $k\in\N$ large and $\eta>0$ small
\begin{equation}
\int_{A_j} \abs{\nabla \varphi_{\eta,k}}^2 \ dx 
\le C  \norm{\nabla u_k}_{L^2(A(\eta,\delta_k))}^2 \left(\gamma^j  + \sum_{l=0}^\infty \gamma^{\abs{l-j}} \int_{A_l} \abs{\nabla\widetilde u_k}^2 \ dx\right),
\end{equation}
where $j\in \N$.
\end{lemma}

\begin{proof}
First we start by bounding
\begin{equation}
\label{EQ: inscdjn98321h834567}
\int_{A_j} \abs{\nabla \varphi_{\eta,k}}^2 \ dx 
\le L \sum_{i=1}^L \int_{A_j} \abs{\nabla \varphi_{\eta,k}^i}^2 \ dx 
\end{equation}
Applying the weighted Wente inequality Lemma F.1 of \cite{DGR22} for $i=1,\dots,L$ to \eqref{EQ: Eq for De phi} we have
\begin{equation}
\label{EQ: uinnewf793213n9nJAUDNB}
\begin{aligned}
\int_{A_j} \abs{\nabla \varphi_{\eta,k}^i}^2 \ dx 
&\le \gamma^j \int_{\C} \abs{\nabla \varphi_{\eta,k}^i}^2 \ dx + C \int_{A(2\eta,\delta_k)} \abs{\nabla\widetilde{\mathrm B}_{\eta,k}^i}^2 dx\ \sum_{l=0}^\infty \gamma^{\abs{l-j}} \int_{A_l} \abs{\nabla \widetilde \Upsilon_{\eta,k}^i}^2 dx.
\end{aligned}
\end{equation}
Using Lemma A.1 of \cite{DLRS25}, \eqref{EQ: inwfo9n2038njJNUDW223d} and Lemma \ref{LEMMA: Linfty bd on na u in the necks unif} we find that
\begin{equation}
\label{EQ: iuenrNU20893n9032unNQ10}
\begin{aligned}
\int_{A(2\eta,\delta_k)} \abs{\nabla\widetilde{\mathrm B}_{\eta,k}^i}^2 dx
\le C \int_{A(\eta,\delta_k)} \abs{\nabla{\mathrm B}_{\eta,k}^i}^2 dx 
\le C \norm{\nabla u_k}_{L^2(A(\eta,\delta_k))}^2
\end{aligned}
\end{equation}
Using \eqref{EQ: iunfr9e902309niJNI93sJM2jms} one has
\begin{equation}
\label{EQ: uiwnefie90931mJNJND12NHIN32in12sdfas}
\begin{aligned}
\int_{A_l} \abs{\nabla \widetilde \Upsilon_{\eta,k}^i}^2 dx
&\le C \int_{A_l} \abs{\nabla \widetilde u_k}^2 dx
\end{aligned}
\end{equation}
as well as with \eqref{EQ: jnfemorjsfJMNi29d91}
\begin{equation}
\label{EQ: uiwfnee89UNJ398nj132AQ}
\norm{\nabla \widetilde \Upsilon_{\eta,k}^i}_{L^2(\C)}
\le C \norm{\nabla u_k}_{L^2(A(\eta,\delta_k))}
\end{equation}
By Wente's inequality applied to \eqref{EQ: Eq for De phi} and the above estimates \eqref{EQ: iuenrNU20893n9032unNQ10}, \eqref{EQ: uiwfnee89UNJ398nj132AQ} we have
\begin{equation}
\label{EQ: nuifwerifndqJNDW093e12e}
\int_{\C} \abs{\nabla \varphi_{\eta,k}^i}^2 \ dx 
\le C \norm{\nabla\widetilde{\mathrm B}_{\eta,k}^i}_{L^2(\C)} \norm{\nabla \widetilde \Upsilon_{\eta,k}^i}_{L^2(\C)}
\le C \norm{\nabla u_k}_{L^2(A(\eta,\delta_k))}^2.
\end{equation}
Hence, with \eqref{EQ: inscdjn98321h834567}, \eqref{EQ: uinnewf793213n9nJAUDNB}, \eqref{EQ: iuenrNU20893n9032unNQ10}, \eqref{EQ: uiwnefie90931mJNJND12NHIN32in12sdfas} and \eqref{EQ: nuifwerifndqJNDW093e12e}
\begin{equation}
\int_{A_j} \abs{\nabla \varphi_{\eta,k}}^2 \ dx 
\le C \gamma^j \norm{\nabla u_k}_{L^2(A(\eta,\delta_k))}^2 + C \norm{\nabla u_k}_{L^2(A(\eta,\delta_k))}^2 \sum_{l=0}^\infty \gamma^{\abs{l-j}} \int_{A_l} \abs{\nabla\widetilde u_k}^2 \ dx.
\end{equation}
\end{proof}

\begin{lemma}[Estimate of $\nabla\mathfrak h_{\eta,k}$]
\label{LEMMA: on the est of na h in necks}
For $k\in\N$ large and $\eta>0$ small there holds
\begin{equation}
\int_{A_j} \abs{\nabla \mathfrak h_{\eta,k}}^2 dz 
\le C \left[ \left( \frac{2^{-j}}{\et}\right)^2 + \left( \frac{\de_k}{2^{-j}\et}\right)^2 \right] \left(\norm{\nabla u_k}_{L^2(A(\eta,\delta_k))}^2 + \norm{\nabla b_{\eta,k}}_{L^{p_k^\prime}(B_1)} \right) ,
\end{equation}
where $j\in\N$ is such that $\frac{2\delta_k}{\eta}\le2^{-j-1}<2^{-j}\le \frac{\eta}{2} $.
\end{lemma}

\begin{proof}
This is the same as Lemma 4.4 in \cite{DLRS25}.
\end{proof}

In the  following lemma we show a sort of entropy condition linking the parameter $p$ and the conformal class of the neck regions. It is actually a consequence of the $\varepsilon$-regularity and the $L^2$-energy quantization as it was already shown in \cite{LZ19}.

\begin{lemma}
\label{LEMMA: Growth of  pk and ln etadelta}
For $\eta>0$ small there holds
\begin{equation}
\lim_{\eta\searrow0}\limsup_{k\rightarrow\infty}\ (p_k-2) \log\left(\frac{\eta^2}{\delta_k}\right)=0.
\end{equation}
\end{lemma}

\begin{proof}
By \eqref{EQ: Cond of bubb conv ji0wr931jJGA} for large $k$ and small $\eta>0$ there holds
\begin{equation}
\label{EQ: iwnfbiININ12081dn012s}
\norm{\nabla v_\infty}_{L^2(\C)} 
\le 2\norm{\nabla v_\infty}_{L^2(B_\frac{1}{\eta})} 
\le 4\norm{\nabla v_k}_{L^2(B_\frac{1}{\eta})} 
=4 \norm{\nabla u_k}_{L^2(B_\frac{\delta_k}{\eta})}.
\end{equation}
Using \eqref{EQ: iwnfbiININ12081dn012s} and applying Lemma A.5 of \cite{DLRS25} to $u_k$ and the radii $r=\delta_k/\eta$, $R=\eta$ one finds
\begin{equation}
(p_k-2) \log\left(\frac{\eta^2}{\delta_k} \right)
\le C \norm{(1+\abs{\nabla u_k }^2)^{\frac{p_k }{2}-1}}_{L^\infty(A(\eta,\delta_k))} \left( \norm{\nabla u_k }_{L^2(A(\eta,\delta_k))}^2 + \eta^2 \right)
\end{equation}
The claim follows by combining Lemma \ref{LEMMA: Linfty bd on na u in the necks unif} and Theorem \ref{THM: L2 en quant of na u}.

\end{proof}

\begin{corollary}
\label{COROLLARY: Limit of na uk to pk minus 2 eq 1}
\begin{equation}
\lim_{k \rightarrow \infty} \norm{\left(1+\abs{\nabla u_k}^2 \right)^{\frac{p_k}{2}-1}}_{L^\infty(\Sigma)}=1.
\end{equation}
\end{corollary}

\begin{proof}
As explained in Lemma 4.2 of \cite{DLRS25} one can bound $\|\nabla u_k\|_{L^\infty(\Sigma)} \le C \delta_k^{-1}$ and hence using Lemma \ref{LEMMA: Growth of  pk and ln etadelta} obtain
\begin{equation}
\norm{\left(1+\abs{\nabla u_k}^2 \right)^{\frac{p_k}{2}-1}}_{L^\infty(\Sigma)}
\le (C\delta_k^{-2}+1)^{\frac{p_k}{2}-1}
= \underbrace{(\delta_k^2+C)^{\frac{p_k}{2}-1}}_{\rightarrow 1} \ \underbrace{\delta_k^{2-p_k}}_{\rightarrow 1}.
\end{equation}
\end{proof}

 The new precise control on the energy of $b_{\eta,k}$ developed in Lemma \ref{LEMMA: First est of nabla b in necks} together with the entropy condition as in Lemma \ref{LEMMA: Growth of  pk and ln etadelta} (coming from \cite{LZ19}) allows to suitably control $\nabla b_{\eta,k}$ in the necks:
\begin{lemma}\label{LEMMA: Control of const of na b in necks}
For $k\in\N$ large and $\eta>0$ small there holds
\begin{equation}
\norm{\nabla b_{\eta,k}}_{L^{p_k^\prime}(A(\eta,\delta_k))}
\le C_{\eta,k},
\end{equation}
where
\begin{equation}
\lim_{\eta\searrow0}\limsup_{k\rightarrow\infty}\ \log\left(\frac{\eta^2}{\delta_k}\right) C_{\eta,k}=0.
\end{equation}
\end{lemma}

\begin{proof}
The claim follows by combining Lemma \ref{LEMMA: First est of nabla b in necks} and Lemma \ref{LEMMA: Growth of  pk and ln etadelta}.
\end{proof}

\begin{lemma} \label{LEMMA: Lower bound on the mass over diadic annulus with p-2}
There exists a constant $C>0$ such that for $k\in\N$ large and $\eta>0$ small the following holds:
For any $j\in\N$ with $\frac{\delta_k}{\eta}<2^{-j}<\eta$ we have
\begin{equation}
(p_k-2) \le C\left( \int_{A_j} \abs{\nabla u_k}^2 dx + 2^{-2j} \right).
\end{equation}
\end{lemma}

\begin{proof}
Combining Lemma A.5 of \cite{DLRS25} and Lemma \ref{LEMMA: Linfty bd on na u in the necks unif} we find
\begin{equation}
(p_k-2) \norm{\nabla u_k}_{L^2(B_{2^{-j-1}})}^2
\le C \left( \int_{A_j} \abs{\nabla u_k}^2 dx + 2^{-j} \right)
\end{equation}
To conclude we use $\norm{\nabla u_k}_{L^2(B_{2^{-j-1}})} \ge \norm{\nabla u_k}_{L^2(B_\frac{\delta_k}{\eta})}$ and also \eqref{EQ: iwnfbiININ12081dn012s}.
\end{proof}

\begin{lemma}\label{LEMMA: Lemma on the start of dyadic est}
There exists a constant $C>0$ such that for $k\in\N$ large and $\eta>0$ small the following holds:
For any $j\in\N$ with $\frac{\delta_k}{\eta}<2^{-j}<\eta$ we have
\begin{equation}
\int_{A_j} \abs{\nabla u_k}^2 dx 
\le C \left( \norm{\nabla b_{\eta,k}}_{L^{p_k^\prime}(A_j)}^2 + \int_{A_j}\abs{\nabla \varphi_{\eta,k}}^2 dx + \int_{A_j}\abs{\nabla \mathfrak h_{\eta,k}}^2 dx+ 2^{-2j(p_k-1)} \right).
\end{equation}
\end{lemma}

\begin{proof}
By Minkowski's inequality and H\"older's inequality
\begin{equation}
\begin{aligned}
\label{EQ: ojwnfe9138inqNIJ21}
\int_{A_j} \abs{\nabla u_k}^{p_k} dx + 2^{-jp_k}
&\ge 2^{1-\frac{p_k}{2}} \left[\left( \int_{A_j} \abs{\nabla u_k}^{p_k} dx\right)^{\frac{2}{p_k}} + 2^{-2j} \right]^{\frac{p_k}{2}} \\
&\ge C \left[ \int_{A_j} \abs{\nabla u_k}^2 dx + 2^{-2j} \right]^{\frac{p_k}{2}} \\
&= C \left[ \int_{A_j} \abs{\nabla u_k}^2 dx + 2^{-2j} \right]^{\frac{p_k^\prime}{2}} \left[ \int_{A_j} \abs{\nabla u_k}^2 dx + 2^{-2j} \right]^{\frac{p_k(p_k-2)}{2(p_k-1)}},
\end{aligned}
\end{equation}
where in the last line we used that $\frac{p_k}{2} = \frac{p_k^\prime}{2} + \frac{p_k(p_k-2)}{2(p_k-1)}$.
By Lemma \ref{LEMMA: Lower bound on the mass over diadic annulus with p-2} we get
\begin{equation}
\label{EQ: Ubnun9e023nfuJN12s1}
\left[ \int_{A_j} \abs{\nabla u_k}^2 dx + 2^{-2j} \right]^{\frac{p_k(p_k-2)}{2(p_k-1)}}
\ge \Big[C (p_k-2) \Big]^{\frac{p_k(p_k-2)}{2(p_k-1)}}
= \underbrace{C^{\frac{p_k(p_k-2)}{2(p_k-1)}}}_{\rightarrow1}  \Big[ \underbrace{(p_k-2)^{(p_k-2)}}_{\rightarrow 1} \Big]^{\frac{p_k}{2(p_k-1)}}.
\end{equation}
Combining \eqref{EQ: ojwnfe9138inqNIJ21} and \eqref{EQ: Ubnun9e023nfuJN12s1}
and using Minkowski's inequality as well as H\"older's inequality one bounds
\begin{equation}
\begin{aligned}
\int_{A_j} \abs{\nabla u_k}^2 dx + 2^{-2j}
&\le C \left[\int_{A_j} \abs{\nabla u_k}^{p_k} dx + 2^{-jp_k} \right]^{\frac{2}{p_k^\prime}} \\
&\le C\left(\int_{A_j} \abs{\nabla u_k}^{p_k} dx  \right)^{\frac{2}{p_k^\prime}} + C\ 2^{-2j(p_k-1)} \\
&\le C\left(\int_{A_j} \abs{X_k}^{p_k^\prime} dx  \right)^{\frac{2}{p_k^\prime}} + C\ 2^{-2j(p_k-1)} \\
&\le C \left( \norm{\nabla b_{\eta,k}}_{L^{p_k^\prime}(A_j)}^2 + \norm{\nabla \varphi_{\eta,k}}_{L^{p_k^\prime}(A_j)}^2 + \norm{\nabla \mathfrak h_{\eta,k}}_{L^{p_k^\prime}(A_j)}^2 + 2^{-2j(p_k-1)} \right) \\
&\le C \left( \norm{\nabla b_{\eta,k}}_{L^{p_k^\prime}(A_j)}^2 + \norm{\nabla \varphi_{\eta,k}}_{L^{2}(A_j)}^2 + \norm{\nabla \mathfrak h_{\eta,k}}_{L^{2}(A_j)}^2 + 2^{-2j(p_k-1)} \right).
\end{aligned}
\end{equation}
\end{proof}

\begin{proof}[Proof (of Theorem \ref{THEOREM: on Pointwise Estimate of the Gradient in the Necks p harm})]
The proof of Theorem \ref{THEOREM: on Pointwise Estimate of the Gradient in the Necks p harm} is very similar to the proof of Theorem 4.1 in \cite{DLRS25}.
We leave details in lengthy and elementary computations out and refer to \cite{DLRS25} for the full details.
Let us introduce
\begin{equation}
\begin{aligned}
a_j & \coloneqq \int_{A_j} \abs{\na\widetilde u_k}^2  dx, \\
b_j & \coloneqq c_0\left[ 2^{-2j(p_k-1)}+ \gamma^j \norm{\nabla u_k}_{L^2(A(\eta,\delta_k))}^2 + \norm{\nabla \mathfrak h_{\eta,k}}_{L^{2}(A_j)}^2+ \norm{\nabla b_{\eta,k}}_{L^{p_k^\prime}(A_j)}^2\right], \\
\varepsilon_0 &= \varepsilon_0(\et,\de_k) \coloneqq  c_0 \int_{A(\eta,\delta_k)} \abs{\nabla u_k}^2  dx.
\end{aligned}
\end{equation}
Combining Lemma \ref{LEMMA: Lemma on the start of dyadic est} and Lemma \ref{LEMMA: Dyadic est of na phi start} one has
\begin{equation}
a_j \le b_j
+ \epv_0 \sum_{l=0}^\infty \gamma^{\abs{l-j}} a_l, \qquad \forall j \in [s_1,s_2]\coloneqq \Bigg[\bigg\lceil-\log_2\left(\frac{\eta}{2}\right)\bigg\rceil, \bigg\lfloor-\log_2\left(\frac{4\delta_k}{\eta}\right)\bigg\rfloor \Bigg].
\end{equation}
Now we apply Lemma G.1 of \cite{DGR22} for some fixed $j \in \{s_1,\dots,s_2 \}$. 
Then for $\gamma<\mu<1$ there exists $C_{\mu, \gamma}>0$ such that
\begin{equation}
\begin{aligned}
\sum_{l=s_1}^{s_2} \mu^{|l-j|} a_{l} 
&= \sum_{l=s_1}^{s_2} \mu^{|l-j|} b_{l} + C_{\mu, \gamma}\ \varepsilon_0 \sum_{l=s_1}^{s_2} \mu^{|l-j|} a_{l} \\
&\hspace{8mm} + C_{\mu, \gamma}\ \varepsilon_0 \left( \mu^{|s_1-1-j|} a_{s_1-1} + \mu^{|s_1-2-j|} a_{s_1-2}
+ \mu^{|s_2+1-j|} a_{s_2+1} + \mu^{|s_2+2-j|} a_{s_2+2}\right),
\end{aligned}
\end{equation}
where we used the fact that $a_l=0$ for any $l \le s_1-3$ or $l \ge s_2+3$ coming from \eqref{EQ: wuinuifnvuUNIDWU0139jr1}.
By Theorem \ref{THM: L2 en quant of na u} 
\begin{equation}
\lim_{\eta\searrow0} \limsup_{k\rightarrow \infty}\ \varepsilon_0(\eta,\delta_k)=0.
\end{equation}
Hence, we can assume that for $\eta>0$ small enough and for $k \in \N$ large enough we have
\begin{equation}
C_{\mu,\gamma} \ \epv_0 < \frac{1}{2}.
\end{equation}
allowing to absorb the sum to the left-hand side
\begin{equation}
\begin{aligned}
\label{EQ: uwnefijnwefh0912ikjmOAJ9201myx}
&\sum_{l=s_1}^{s_2} \mu^{|l-j|} a_{l} \\
&\hspace{7mm} \le C \sum_{l=s_1}^{s_2} \mu^{|l-j|} b_{l} + C \ \varepsilon_0 \left( \mu^{j-s_1} a_{s_1-1} + \mu^{j-s_1} a_{s_1-2}
+ \mu^{s_2-j} a_{s_2+1} + \mu^{s_2-j} a_{s_2+2}\right) \\
&\hspace{7mm} \le C \sum_{l=s_1}^{s_2} \mu^{|l-j|} b_{l} + C \ \varepsilon_0 \left( \mu^{j-s_1} + \mu^{s_2-j}\right)\norm{\nabla u_k}_{L^2(A(\eta,\delta_k))}^2,
\end{aligned}
\end{equation}
where in the last line we used that for any $i$ one has $a_i \le \norm{\na\widetilde u_k}_{L^2(\C)}^2 \le C\norm{\nabla u_k}_{L^2(A(\eta,\delta_k))}^2$.
We introduce $\beta \coloneqq - \log_2 \mu\in (0,-\log_2 \gamma)\subset(0,1)$ such that
\begin{equation}
\mu = 2^{-\beta}.
\end{equation}
Now we focus on the bound of the expression
\begin{equation}
\begin{aligned}
\label{EQ: we832hrnu31rHBI121w12lexs}
&\sum_{l=s_1}^{s_2} \mu^{|l-j|} b_{l}\\
&\hspace{12mm} = c_0\sum_{l=s_1}^{s_2} \mu^{|l-j|} \left[ 2^{-2l(p_k-1)} + \gamma^l \norm{\nabla u_k}_{L^2(A(\eta,\delta_k))}^2 + \norm{\nabla \mathfrak h_{\eta,k}}_{L^{2}(A_l)}^2+ \norm{\nabla b_{\eta,k}}_{L^{p_k^\prime}(A_l)}^2\right].
\end{aligned}
\end{equation}
\textbf{1.)}
\begin{equation}
\begin{aligned}
\sum_{l=s_1}^{s_2} \mu^{|l-j|} 2^{-2l(p_k-1)}
\le \mu^{j}\sum_{l=s_1}^{j} \mu^{-l} 2^{-2l}
+ \mu^{-j} \sum_{l=j+1}^{s_2} \mu^{l} 2^{-2l} \le \left( \frac{1}{4\mu}\right)^{s_1-1}  2^{-\beta j} + 2^{-s_1} 2^{-j}.
\end{aligned}
\end{equation}
\textbf{2.)}
\begin{equation}
\begin{aligned}
\sum_{l=s_1}^{s_2} \mu^{\abs{l-j}} \gamma^l 
&= \mu^{j} \sum_{l=s_1}^{j} \left(\frac{\gamma}{\mu} \right)^l + \mu^{-j} \sum_{l=j+1}^{s_2} \left(\mu\ga\right)^l 
\le \mu^{j}  + \gamma^j \le 2\mu^j =2\ 2^{-\beta j}.
\end{aligned}
\end{equation}
\textbf{3.)}
With Lemma \ref{LEMMA: on the est of na h in necks} we get
\begin{equation}
\begin{aligned}
&\hspace{-10mm} \sum_{l=s_1}^{s_2} \mu^{|l-j|} \norm{\nabla \mathfrak h_{\eta,k}}_{L^{2}(A_l)}^2 \\
&\le C \left(\norm{\nabla u_k}_{L^2(A(\eta,\delta_k))}^2 + \norm{\nabla b_{\eta,k}}_{L^{p_k^\prime}(B_1)} \right) \sum_{l=s_1}^{s_2} \mu^{|l-j|} \left[ \left( \frac{2^{-l}}{\et}\right)^2 + \left( \frac{\de_k}{2^{-l}\et}\right)^2 \right] ,
\end{aligned}
\end{equation}
and compute
\begin{equation}
\begin{aligned}
&\hspace{-12mm}\sum_{l=s_1}^{s_2} \mu^{\abs{l-j}} \left[ \left( \frac{2^{-l}}{\et}\right)^2 + \left( \frac{\de_k}{2^{-l}\et}\right)^2 \right] \\
&= \frac{1}{\et^2} \left[ \sum_{l=s_1}^{j} \mu^{j}\left( \left( \frac{1}{4\mu} \right)^l + \left( \frac{4}{\mu} \right)^l \de_k^2 \right)
+  \sum_{l=j+1}^{s_2} \mu^{-j}\left( \left( \frac{\mu}{4} \right)^l + \left( 4\mu \right)^l \de_k^2 \right) \right] \\
&\le C \left[ \left( \frac{2^{-j}}{\et} \right)^\be + \left( \frac{\de_k}{2^{-j}\et} \right)^\be \right].
\end{aligned}
\end{equation}
where in the last line we used that $\mu = 2^{-\beta}\in (\frac{1}{4},1)$, $\be <2$, $\frac{2^{-j}}{2\et}\le 1$, $\frac{\de_k}{2^{-j}\et}\le 1$, $\frac{2^{-s_1}}{\eta}\le C$ and $\frac{\delta_k}{2^{-s_1}\eta}\le C$.
\\ \textbf{4.)}
We bound using \eqref{EQ: HIninfwjei921ejir2}
\begin{equation}
\begin{aligned}
\label{EQ: 0921niwebfweijwefn8913}
\sum_{l=s_1}^{s_2} \mu^{|l-j|}  \norm{\nabla b_{\eta,k}}_{L^{p_k^\prime}(A_l)}^2
&\le \norm{\nabla b_{\eta,k}}_{L^{p_k^\prime}(A(\eta,\delta_k))}^{2-p_k^\prime} \sum_{l=s_1}^{s_2} \mu^{|l-j|}  \norm{\nabla b_{\eta,k}}_{L^{p_k^\prime}(A_l)}^{p_k^\prime} \\
&\le (C (p_k-2))^{2-p_k^\prime} \norm{\nabla b_{\eta,k}}_{L^{p_k^\prime}(A(\eta,\delta_k))}^{p_k^\prime},
\end{aligned}
\end{equation}
where in the last line we used additivity of the integral.
As $2-p_k^\prime=\frac{p_k-2}{p_k-1}$ we have
\begin{equation}
(C (p_k-2))^{2-p_k^\prime} 
\le C (p_k-2)^{2-p_k^\prime}
= C (\underbrace{(p_k-2)^{p_k-2}}_{\rightarrow 1})^\frac{1}{p_k-1}
\le C
\end{equation}
and also
\begin{equation}
\begin{aligned}
\norm{\nabla b_{\eta,k}}_{L^{p_k^\prime}(A(\eta,\delta_k))}^{p_k^\prime}
& \le \Big[\norm{\nabla b_{\eta,k}}_{L^{p_k^\prime}(A(\eta,\delta_k))} + (p_k-2)\Big]^{p_k^\prime} \\
&= \Big[\norm{\nabla b_{\eta,k}}_{L^{p_k^\prime}(A(\eta,\delta_k))} + (p_k-2)\Big]^{2} \ \Big[\norm{\nabla b_{\eta,k}}_{L^{p_k^\prime}(A(\eta,\delta_k))} + (p_k-2)\Big]^{\overbrace{p_k^\prime-2}^{<0}} \\
&\le C \Big[\norm{\nabla b_{\eta,k}}_{L^{p_k^\prime}(A(\eta,\delta_k))}^2 + (p_k-2)^2\Big] \ \underbrace{\Big[p_k-2\Big]^{p_k^\prime-2}}_{\rightarrow 1}.
\end{aligned}
\end{equation}
With \eqref{EQ: 0921niwebfweijwefn8913} we get
\begin{equation}
\begin{aligned}
\sum_{l=s_1}^{s_2} \mu^{|l-j|}  \norm{\nabla b_{\eta,k}}_{L^{p_k^\prime}(A_l)}^2
\le C \Big[\norm{\nabla b_{\eta,k}}_{L^{p_k^\prime}(A(\eta,\delta_k))}^2 + (p_k-2)^2\Big] \eqqcolon \left(C_{\eta,k}\right)^2,
\end{aligned}
\end{equation}
where with Lemma \ref{LEMMA: Control of const of na b in necks} and Lemma \ref{LEMMA: Growth of  pk and ln etadelta} one has $\lim_{\eta\searrow0}\limsup_{k\rightarrow\infty}\ \log\left(\frac{\eta^2}{\delta_k}\right) C_{\eta,k}=0$. \\
Putting these bounds \textbf{1.) - 4.)} together and with \eqref{EQ: uwnefijnwefh0912ikjmOAJ9201myx}, \eqref{EQ: we832hrnu31rHBI121w12lexs} we find
\begin{equation}
\label{EQ: INuhnfue91uqfwUNqw}
\begin{aligned}
\int_{A_j} \abs{\nabla u_k}^2 dx 
&=a_j \le \sum_{l=s_1}^{s_2} \mu^{|l-j|} a_{l} \\
&\le \left(\norm{\nabla u_k}_{L^2(A(\eta,\delta_k))}^2 + \norm{\nabla b_{\eta,k}}_{L^{p_k^\prime}(B_1)} \right) \Bigg[ \left( \frac{2^{-j}}{\et} \right)^\be + \left( \frac{\de_k}{2^{-j}\et} \right)^\be \Bigg] \\
&\hspace{35mm} + C \norm{\nabla u_k}_{L^2(A(\eta,\delta_k))}^2 \Bigg[ \mu^{j-s_1} + \mu^{s_2-j} + 2^{-\beta j} \Bigg] \\
&\hspace{45mm} + C \Bigg[ \left( \frac{1}{4\mu}\right)^{s_1} \ 2^{-\beta j} + 2^{-s_1} 2^{-j} + \left(C_{\eta,k}\right)^2 \Bigg].
\end{aligned}
\end{equation}
One has the following:
\begin{equation}
\begin{gathered}
\mu^{j-s_1} =\left(2^{-\beta(j-s_1)} \right) = \left(2^{-j} 2^{s_1} \right)^\beta 
\le C \left( \frac{2^{-j}}{\eta} \right)^\beta \\
\mu^{s_2-j} =\left(2^{-\beta(s_2-j)} \right) = \left(2^{j}\ 2^{-s_2} \right)^\beta 
\le C \left( \frac{\delta_k}{2^{-j}\eta} \right)^\beta \\
2^{-\beta j}
\le \left( \frac{2^{-j}}{\eta} \right)^\beta \\
\left( \frac{1}{4\mu}\right)^{s_1} \ 2^{-\beta j}
\le C\left( 2^{\beta-2}\right)^{-\log_2(\eta)} 2^{-\beta j}
\le C \eta^{2-\beta} 2^{-\beta j}
=  C \eta^{2} \left( \frac{2^{-j}}{\eta} \right)^\beta \\
2^{-s_1} 2^{-j} \le C \ 2^{\log_2 \eta} 2^{-j}
\le C \ \eta \left( \frac{2^{-j}}{\eta} \right)
\le C \ \eta \left( \frac{2^{-j}}{\eta} \right)^\beta
\end{gathered}
\end{equation}
Going back to \eqref{EQ: INuhnfue91uqfwUNqw} we have found
\begin{equation}
\label{EQ: ijnvewiUN1298d1}
\begin{aligned}
&\int_{A_j} \abs{\nabla u_k}^2 dx \\
&\hspace{12mm}\le C \left( \norm{\nabla u_k}_{L^2(A(\eta,\delta_k))}^2+ \norm{\nabla b_{\eta,k}}_{L^{p_k^\prime}(B_1)} + \eta+\eta^2 \right)
\left( \left( \frac{2^{-j}}{\et} \right)^\be + \left( \frac{\de_k}{2^{-j}\et} \right)^\be \right) + \left(C_{\eta,k}\right)^2
\end{aligned}
\end{equation}
Let $x \in A_j$.
Put $r_x = \abs{x}/4$.
One has $B_{r_x}(x) \subset A_{j-1} \cup A_{j} \cup A_{j+1}$.
By $\varepsilon$-regularity Lemma \ref{LEMMA: epsilon reg sacks uhlenbeck for p harm} we can bound
\begin{equation}
\label{EQ: iwunfeiuNnquf39831dasdas}
\begin{aligned}
\abs{x}^2 \abs{\nabla u_k(x)}^2 
&= 2^6 \left(\frac{r_x}{2}\right)^2 \abs{\nabla u_k(x)}^2 \le 2^6 \left(\frac{r_x}{2}\right)^2 \norm{\nabla u_k}_{L^{\infty}(B_{r_x/2}(x))}^2 \le C \ \norm{\nabla u_k}_{L^2(B_{r_x}(x))}^2 \\
&\le C \left( \norm{\nabla u_k}_{L^2(A_{j-1})}^2 + \norm{\nabla u_k}_{L^2(A_{j})}^2 + \norm{\nabla u_k}_{L^2(A_{j+1})}^2 \right).
\end{aligned}
\end{equation}
Combining \eqref{EQ: ijnvewiUN1298d1} and \eqref{EQ: iwunfeiuNnquf39831dasdas} with the fact that $2^{-j-1}\le \abs{x}\le 2^{-j}$ we get.
\begin{equation}
\abs{x}^2 \abs{\nabla u_k(x)}^2 
\le C \left( \norm{\nabla u_k}_{L^2(A(\eta,\delta_k))}^2+ \norm{\nabla b_{\eta,k}}_{L^{p_k^\prime}(B_1)} + \eta+\eta^2 \right)
\left( \left( \frac{\abs{x}}{\et} \right)^\be + \left( \frac{\de_k}{\abs{x}\et} \right)^\be \right) + \left(C_{\eta,k}\right)^2
\end{equation}
With Theorem \ref{THM: L2 en quant of na u}, Lemma \ref{LEMMA: First est of nabla b in necks} and Lemma \ref{LEMMA: Control of const of na b in necks} for all $x \in A_j$ we have
\begin{equation}
\label{EQ: jwiefnuIN301nidqw}
\abs{x}^2 \abs{\nabla u_k(x)}^2 \le \left[ \left( \frac{\abs{x}}{\et} \right)^\be + \left( \frac{\de_k}{\et \abs{x}} \right)^\be \right]\boldsymbol{\epsilon}_{\eta, \delta_k} +  \boldsymbol{\mathrm c}_{\eta, \delta_k},
\end{equation}
where
\begin{equation}
\lim_{\eta\searrow0}\limsup_{k\rightarrow\infty}\boldsymbol{\epsilon}_{\eta, \delta_k}=0,
\qquad \text{and} \qquad
\lim_{\eta\searrow0}\limsup_{k\rightarrow\infty}\boldsymbol{\mathrm c}_{\eta, \delta_k} \log^2\left(\frac{\eta^2}{\delta_k}\right) =0.
\end{equation}
Let now $x\in A\left(\frac{\eta}{4},\delta_k\right)$.
Then we can find some $j\in \N$ such that $2^{-j-1}\le \abs{x}\le 2^{-j}$.
But then $\frac{4\delta_k}{\eta} \le \abs{x} \le 2^{-j} \le 2 \abs{x} \le \frac{\eta}{2}$.
Therefore $j \in \{s_1, \dots,s_2 \}$ and estimate \eqref{EQ: jwiefnuIN301nidqw} is valid for $x\in A\left(\frac{\eta}{4},\delta_k\right)$.
This completes the proof.
\end{proof}

\vspace{10mm}
\section{Stability of the Morse Index}

In this section we finally show the upper semicontinuity of the Morse index plus nullity for Sacks-Uhlenbeck sequences to a homogeneous manifold, more precisely
\begin{theorem}
\label{THM: usc morse index sacks-uhlenbeck with one bubble}
For $k\in \N $ large there holds
\begin{equation}
\operatorname{Ind}_{E_{p_k}}(u_k) + \operatorname{Null}_{E_{p_k}}(u_k)
\le \operatorname{Ind}_{E}(u_\infty) + \operatorname{Null}_{E}(u_\infty) + \operatorname{Ind}_{E}(v_\infty) + \operatorname{Null}_{E}(  v_\infty)
\end{equation}
\end{theorem}
We adapt the strategy introduced in \cite{DGR22} and closely follow \cite{DLRS25}.
Let us briefly explain what this is. 
First, we show that the necks are not contributing to the negativity of the second variation. This we do by combining the pointwise control as in estimate \eqref{EQ: 283hfunv0lpdnu13ef6zh1} and a weighted Poincar\`e inequality (Lemma A.9 of \cite{DLRS25}).
Second, we use Sylvester's law of inertia to change to a different measure incorporating the weights obtained in estimate \eqref{EQ: 283hfunv0lpdnu13ef6zh1}.
Finally, we apply spectral theory to the Jacobi operator of the second variation.
The result follows by combining these techniques.

\subsection{Positive contribution of the Necks}

In this section we prove that any variation supported in the neck region evaluates positively in the quadratic form.
More concrete:

\begin{theorem}
\label{THM: JDNW892fejeujfwef32}
For every $\be \in \left(0,\log_2 (3/2) \right)$ there exists some constant $\overline\kappa>0$ such that for $k\in \N$ large and $\eta >0$ small one has
\begin{equation}
\forall w \in V_{u_k} : \
(w=0 \text{ in } \Sigma \setminus A(\eta,\delta_k))
\Rightarrow
Q_{u_k}(w) \ge \overline\kappa \int_\Sigma \abs{w}^2 \omega_{\eta,k} \ dvol_h \ge 0,
\end{equation} 
where the weight function is given by
\begin{equation}
\omega_{\eta,k}=
\begin{cases}
\frac{1}{\abs{x}^2} \left[ \frac{\abs{x}^\beta}{\eta^\beta} + \frac{\delta_k^\beta}{\eta^\beta \abs{x}^\beta} + \frac{1}{\log^2\left(\frac{\eta^2}{\delta_k} \right)} \right] &\text{ if } x \in A(\eta,\delta_k), \\
\frac{1}{\eta^2} \left[1+  \frac{\delta_k^\beta}{\eta^{2\beta}} + \frac{1}{\log^2\left(\frac{\eta^2}{\delta_k} \right)} \right] &\text{ if } x \in \Sigma \setminus B_\eta, \\
\frac{\eta^2}{\delta_k^2} \left[\frac{(1+\eta^2)^2}{\eta^4\left(1+\delta_k^{-2} \abs{x}^2\right)^2} +  \frac{\delta_k^\beta}{\eta^{2\beta}} + \frac{1}{\log^2\left(\frac{\eta^2}{\delta_k} \right)} \right] &\text{ if } x \in B_{\delta_k/\eta}.
\end{cases}
\end{equation}
\end{theorem}

\begin{proof}
This result follows from the pointwise control on the gradient in the necks coming from Theorem \ref{THEOREM: on Pointwise Estimate of the Gradient in the Necks p harm}.
The proof of Theorem \ref{THM: JDNW892fejeujfwef32} is the same as the proof of Theorem 5.2 in \cite{DLRS25}. One simply needs to use the bound
\begin{equation}
\Big|\mathbb I_{u_k}(\nabla u_k, \nabla u_k) \cdot \mathbb I_{u_k}(w,w)\Big|
\le C \abs{\nabla u}^2 \abs{w}^2
\end{equation}
of the term appearing in the second variation of the energy in Definition \ref{DEFINITION: Morse index and Nullityof p-harmonic maps} and follow the proof of Theorem 5.2 in \cite{DLRS25}.
\end{proof}

\subsection{The Diagonalization of $Q_{u_k}$ with respect to the Weights $\omega_{\eta,k}$}
Let $\mathrm{n}_1,\dots,\mathrm{n}_{m-n} \in \Gamma\big((T\mathcal N)^\perp\big)$ be an orthonormal frame of the normal bundle of $\mathcal N$.
Define
\begin{equation}
\label{EQ: iwenf8932injIUJN128ee12}
S_{u_{k}}(\nabla u_{k}) \coloneqq \sum_{j=1}^{m-n} \Big\langle \mathbb I_{u_k}(\nabla u_k,\nabla u_k), \mathrm n_j(u_k) \Big\rangle\ D(\mathrm n_{j})_{u_k}
\end{equation}
such that for any tangent vector fields $X,Y \in \Gamma (T\mathcal N)$ we have
\begin{equation}
\mathbb I_{u_k}(\nabla u_k, \nabla u_k) \cdot \mathbb I_{u_k}(X,Y)
= (S_{u_{k}}(\nabla u_{k}) X) \cdot Y,
\end{equation}
and the pointwise bound
\begin{equation}
\abs{S_{u_{k}}(\nabla u_{k}) X}
\le C \abs{\nabla u_k}^2 \abs{X}.
\end{equation}
Consider the inner product
\begin{equation}
\langle w,v \rangle_{\omega_{\eta,k}} \coloneqq \int_{\Sigma} w \cdot v \ \omega_{\eta,k} \ dvol_{\Sigma}.
\end{equation}
\noindent
Then we look for the self-adjoint Jacobi operator with respect to $\langle \cdot,\cdot \rangle_{\omega_{\eta,k}}$ of the quadratic form $Q_{u_k}$.
Let us introduce the operator
\begin{equation}
\begin{aligned}
\mathcal L_{\eta,k}(w) 
&\coloneqq \omega_{\eta,k}^{-1}\ P_{u_k} \Bigg[ \left(-p_k\ (p_k-2) \divergence \left( \left( 1+\abs{\nabla u_k}^2\right)^{p_k/2-2} \left(\nabla u_k \cdot \nabla w \right)\ \nabla u_k \right) \right) \\
&\hspace{10mm} - p_k \divergence\left( \left( 1+\abs{\nabla u_k}^2\right)^{p_k/2-1} \nabla w \right) - p_k \left( 1+\abs{\nabla u_k}^2\right)^{p_k/2-1} S_{u_{k}}(\nabla u_{k}) w\Bigg],
\end{aligned}
\end{equation}
where $P_{u_k}(x): \R^{m}\rightarrow T_{u_k(x)} \mathcal N$ is the orthogonal projection.
Then integrating by parts we have the formula
\begin{equation}
\label{EQ: iwubfnIUBPN239fn321}
Q_{u_k}(w)= \langle \mathcal L_{\eta,k} w,w \rangle_{\omega_{\eta,k}}. 
\end{equation}
Note that by construction $\mathcal L_{\eta,k}$ is self-adjoint with respect to the inner product $\langle \cdot ,\cdot \rangle_{\omega_{\eta,k}}$, i.e.
\begin{equation}\label{EQ: wijenfuNIUD138en9qd}
\langle \mathcal L_{\eta,k} w,v \rangle_{\omega_{\eta,k}} = \langle w, \mathcal L_{\eta,k} v \rangle_{\omega_{\eta,k}} .
\end{equation}
Recall the definition of $V_{u_k}$ in \eqref{EQ: wigunUBNI8193hf23f} and consider also the larger space
\begin{equation}
U_{u_k} = \left\{ w \in L^{2}_{\omega_{\eta,k}}(\Sigma; \R^{m}) \forwhich w(x) \in T_{u_k(x)} \mathcal N, \quad \text{for a.e. } x\in\Sigma \right\}.
\end{equation}

\begin{lemma}[Spectral Decomposition]
\label{LEMMA: Spectral decomp of L}
There exists a Hilbert basis of the space $(U_{u_k}, \langle\cdot,\cdot\rangle_{\omega_{\eta,k}})$ of eigenfunctions of the operator $\mathcal L_{\eta,k}$ and the eigenvalues of $\mathcal L_{\eta,k}$ satisfy 
\begin{equation}
\lambda_1<\lambda_2<\lambda_3\ldots \rightarrow\infty.
\end{equation}
Furthermore, one has the orthogonal decomposition 
\begin{equation}
U_{u_k} = \bigoplus_{\lambda \in \Lambda_{\eta,k}}\mathcal E_{\eta,k}(\lambda),
\end{equation}
where 
\begin{equation}
\mathcal E_{\eta,k}(\lambda) \coloneqq \{ w \in V_{u_k} \ ; \ \mathcal L_{\eta,k}(w)=\lambda w \},
\qquad \Lambda_{\eta,k} \coloneqq\big\{ \lambda\in \R \ ; \ \mathcal E_{\eta,k}(\lambda) \setminus\{0\}  \neq \emptyset \big\}
\end{equation}
\end{lemma}

\begin{proof}
This result is obtained by using the spectral theory for compact self-adjoint operators on a Hilbert space. It is the same as in Lemma 5.3 of \cite{DLRS25}, but one has to incorporate the bound
\begin{equation}
\Big|\mathbb I_{u_k}(\nabla u_k, \nabla u_k) \cdot \mathbb I_{u_k}(w,w)\Big|
\le C \abs{\nabla u}^2 \abs{w}^2.
\end{equation}
\end{proof}

\begin{lemma}[Sylvester Law of Inertia]
\label{LEMMA: Ind plus Null eq dim of eigenspaces}
\begin{equation}
\operatorname{Ind}(u_k) + \operatorname{Null}(u_k) = \operatorname{dim}\left( \bigoplus_{\lambda\le 0} \mathcal E_{\eta,k}(\lambda) \right)
\end{equation}
\end{lemma}

\begin{proof}
This is a direct consequence of the spectral decomposition in Lemma \ref{LEMMA: Spectral decomp of L}.
For all details see the proof of Lemma 5.4 in \cite{DLRS25}.
\end{proof}

Set
\begin{equation}
\label{EQ: JIN98j013dmio1das}
\mu_{\eta,k}\coloneqq \norm{\frac{\abs{\nabla u_k}^2}{\omega_{\eta,k}}}_{L^\infty(\Sigma)}.
\end{equation}
Then one has

\begin{lemma}
\label{LEMMA: prop on the seq mu eta k}
\begin{equation}
\begin{aligned}
&\hspace{-6mm} \exists \eta_0>0: \ \exists k_0>0: \ \exists C>0: \\
&i)\ \mu_0\coloneqq \sup_{\eta\in(0,\eta_0)} \sup_{k\ge k_0} \mu_{\eta,k}< \infty,  \\
&ii)\ \lim_{\eta \searrow0} \limsup_{k \rightarrow \infty} \mu_{\eta,k}=0, \\
&iii)\ \forall\eta \in(0,\eta_0): \forall k \ge k_0: \inf \Lambda_{\eta,k} \ge -C\ \mu_{\eta,k}\ge-C\ \mu_0.
\end{aligned}
\end{equation}
\end{lemma}

\begin{proof}
\textit{i) \& ii):}
We decompose
\begin{equation}
\label{EQ: 2iufnujfnuiwseUJN}
\mu_{\eta,k} \le \norm{\frac{\abs{\nabla u_k}^2}{\omega_{\eta,k}}}_{L^\infty(\Sigma\setminus B_{\eta} )} + \norm{\frac{\abs{\nabla u_k}^2}{\omega_{\eta,k}}}_{L^\infty(B_\eta\setminus B_{\delta_k/\eta} )} + \norm{\frac{\abs{\nabla u_k}^2}{\omega_{\eta,k}}}_{L^\infty(B_{\delta_k/\eta} ))}
\end{equation}
 Note that by Theorem \ref{THEOREM: on Pointwise Estimate of the Gradient in the Necks p harm} we have
\begin{equation}
\label{EQ: i29jtgjijsgrisdfNJD}
\norm{\frac{\abs{\nabla u_k}^2}{\omega_{\eta,k}}}_{L^\infty(B_\eta\setminus B_{\delta_k/\eta} )} \leq \boldsymbol{\epsilon}_{\eta, \delta_k} + \boldsymbol{\mathrm c}_{\eta, \delta_k} \log^2\left(\frac{\eta^2}{\delta_k}\right)\longrightarrow0,
\end{equation}
as $k\rightarrow\infty,\eta\searrow0$.
Furthermore,
\begin{equation}
\label{EQ: 93j4ugj3uNJUSDJU}
\lim_{\eta \searrow0} \limsup_{k \rightarrow \infty}
\norm{\frac{\abs{\nabla u_k}^2}{\omega_{\eta,k}}}_{L^\infty(\Sigma\setminus B_{\eta} )}
\le \lim_{\eta \searrow0} \limsup_{k \rightarrow \infty} \eta^2 \norm{\nabla u_k}_{L^\infty(\Sigma\setminus B_{\eta} )}^2=0.
\end{equation}
Recall that due to the point removability theorem and the stereographic projection one has that
\begin{equation}
\label{EQ: uiwpgjnpi324gnu2f3njJUNDW}
\abs{\nabla v_\infty (y)}^2 \le C\frac{1}{\left(1+\abs{y}^2\right)^2}.
\end{equation}
For $x \in B_{\delta_k/\eta}$ we estimate
\begin{equation}
\label{EQ: NJUWij2f2ifjm2f12}
\begin{aligned}
\frac{\abs{\nabla u_k(x)}^2}{\omega_{\eta,k}(x)}
&\le \frac{\delta_k^2\ \eta^2\left(1+\delta_k^{-2} \abs{x}^2\right)^2}{(1+\eta^2)^2}  \abs{\nabla u_k(x)}^2  = \frac{\eta^2}{(1+\eta^2)^2} \norm{\abs{\nabla v_k(y)}^2 \left(1+ \abs{y}^2\right)^2}_{L^\infty(B_{1/\eta})}.
\end{aligned}
\end{equation}
Note that due to uniform convergence 
\begin{equation}
\limsup_{k \rightarrow \infty} \norm{\abs{\nabla v_k(y)}^2 \left(1+ \abs{y}^2\right)^2}_{L^\infty(B_{1/\eta})}
= \norm{\abs{\nabla v_\infty(y)}^2 \left(1+ \abs{y}^2\right)^2}_{L^\infty(B_{1/\eta})} \le C,
\end{equation}
where we used \eqref{EQ: uiwpgjnpi324gnu2f3njJUNDW} in the last inequality. 
Going back to \eqref{EQ: NJUWij2f2ifjm2f12} this allows to finally get
\begin{equation}
\label{EQ: 2j3mi2g4njnJNJDW}
\lim_{\eta\searrow0}\limsup_{k \rightarrow \infty}\norm{\frac{\abs{\nabla u_k}^2}{\omega_{\eta,k}}}_{L^\infty(B_{\delta_k/\eta} ))}
\le C\lim_{\eta\searrow0}\frac{\eta^2}{(1+\eta^2)^2}=0.
\end{equation}
Going back to \eqref{EQ: 2iufnujfnuiwseUJN} and combining \eqref{EQ: i29jtgjijsgrisdfNJD}, \eqref{EQ: 93j4ugj3uNJUSDJU}, \eqref{EQ: 2j3mi2g4njnJNJDW} we conclude \textit{i)} and \textit{ii)}.

\noindent
\textit{iii):} 
Let $\lambda \in \Lambda_{\eta,k}$.
Then there exists an eigenvector $0\ne w \in V_{u_k}$ of $\mathcal L_{\eta,k}$ corresponding to the eigenvalue $\lambda$, i.e. $\mathcal L_{\eta,k}(w)=\lambda w$.
We get
\begin{equation}
\begin{aligned}
\label{EQ: unu9nJIOIDWYdMwM1289jne}
\lambda \langle w,w \rangle_{\omega_{\eta,k}}
=  \langle \mathcal L_{\eta,k}(w),w \rangle_{\omega_{\eta,k}} =Q_{u_k}(w) 
\ge  -C\int_\Sigma \left( 1+\abs{\nabla u_k}^2\right)^{p_k/2-1} \abs{\nabla u_k}^2 \abs{w}^2 \ dvol_\Sigma
\end{aligned}
\end{equation}
With \eqref{EQ: unu9nJIOIDWYdMwM1289jne}, Lemma \ref{LEMMA: Linfty bd on na u in the necks unif} and \eqref{EQ: JIN98j013dmio1das} we get 
\begin{equation}
\begin{aligned}
\lambda \langle w,w \rangle_{\omega_{\eta,k}} 
\ge - C\ \mu_{\eta,k} \langle w,w \rangle_{\omega_{\eta,k}}.
\end{aligned}
\end{equation}
This completes the proof of the lemma.
\end{proof}

In the following we focus on the limiting maps $u_\infty:\Sigma\rightarrow S^n$ and $v_\infty: \C \rightarrow S^n$ as appearing in Definition \ref{Definition: Bubble tree convergence of p harm map with one bubble}.
We proceed analogous to \cite{DLRS25} and \cite{DGR22}.
We compute for $w\in V_{u_\infty}$ integrating by parts
\begin{equation}
\begin{aligned}
Q_{u_\infty}(w) 
=2 \int_{\Sigma} \left( -\Delta w - S_{u_{\infty}}(\nabla u_{\infty})w \right) \cdot w\ dvol_\Sigma.
\end{aligned}
\end{equation}
Note that for any fixed $\eta>0$ we have the pointwise limit
\begin{equation}
\omega_{\eta,k}(x) \rightarrow \omega_{\eta,\infty}(x) \coloneqq \begin{cases}
\frac{1}{\eta^2}, &\text{ if } x\in \Sigma \setminus B_{\eta} \\
\frac{1}{\eta^\beta \abs{x}^{2-\beta}}, &\text{ if } x\in B_{\eta}
\end{cases}, \qquad \text{ as } k \rightarrow \infty.
\end{equation}
We introduce 
\begin{equation}
\mathcal L_{\eta,\infty}: V_{u_{\infty}} \rightarrow V_{u_{\infty}};
\quad\mathcal L_{\eta,\infty}(w) \coloneqq 2\ P_{u_\infty} \left(\omega_{\eta,\infty}^{-1}(- \Delta w  - S_{u_{\infty}}(\nabla u_{\infty})w  \right), 
\end{equation}
such that
\begin{equation}
Q_{u_\infty}(w)= \langle \mathcal L_{\eta,\infty} w, w \rangle_{\omega_{\eta,\infty}},
\end{equation}
where we used 
\begin{equation}
\langle w,v \rangle_{\omega_{\eta,\infty}} \coloneqq \int_{\Sigma} w \cdot v \ \omega_{\eta,\infty} \ dvol_{\Sigma}.
\end{equation}
As above a simple integration by parts shows that
\begin{equation}
\begin{aligned}
Q_{v_\infty}(w)
&= 2\int_\C \left(- \Delta w  - S_{v_{\infty}}(\nabla v_{\infty})w \right) \cdot w  \ dz,
\end{aligned}
\end{equation}
where $S_{v_{\infty}}(\nabla v_{\infty})$ is defined similar to $S_{u_{\infty}}(\nabla u_{\infty})$.
Let $v_k(z) \coloneqq u_k(\delta_k z)$ as in Definition \ref{Definition: Bubble tree convergence of p harm map with one bubble}.
With a change of variables
\begin{equation}
\int_{B_{\eta}} \abs{\nabla u_k(x)}^2 \omega_{\eta,k}(x) \ dx
= \int_{B_{\frac{\eta}{\delta_k}}} \abs{\nabla v_k(z)}^2 \delta_k^2 \ \omega_{\eta,k}(\delta_k z) \ dz
\end{equation}
motivating the definition of
\begin{equation}
\widehat\omega_{\eta,k}(z) \coloneqq \delta_k^2 \ \omega_{\eta,k}(\delta_k z), \qquad z \in B_{\frac{\eta}{\delta_k}}.
\end{equation}
One has the pointwise limit
\begin{equation}
\widehat \omega_{\eta,k}(z) = \delta_k^2 \ \omega_{\eta,k}(\delta_k z) \rightarrow \widehat \omega_{\eta,\infty} (z) \coloneqq \begin{cases}
\frac{1}{\eta^\beta} \frac{1}{\abs{z}^{2+\beta}} , &\text{ if } z\in \C \setminus B_{1/\eta} \\
\frac{1}{\eta^2} \frac{(1+\eta^2)^2}{(1+\abs{z}^2)^2}, &\text{ if } z\in B_{1/\eta}
\end{cases},
\qquad \text{ as } k \rightarrow \infty.
\end{equation}
We introduce 
\begin{equation}
\widehat{\mathcal L}_{\eta,\infty}: V_{v_{\infty}} \rightarrow V_{v_{\infty}}; 
\qquad\widehat{\mathcal L}_{\eta,\infty}(w) \coloneqq 2\ P_{v_\infty} \left(\widehat\omega_{\eta,\infty}^{-1}(- \Delta w  -S_{v_{\infty}}(\nabla v_{\infty})w) \right),
\end{equation}
such that 
\begin{equation}
Q_{v_\infty}(w)= \langle \widehat{\mathcal L}_{\eta,\infty} w, w \rangle_{\widehat\omega_{\eta,\infty}},
\end{equation}
where we used 
\begin{equation}
\langle w,v \rangle_{\widehat\omega_{\eta,\infty}} \coloneqq \int_{\C} w \cdot v \ \widehat \omega_{\eta,\infty} \ dz.
\end{equation}
In the following let 
\begin{equation}
St: S^2 \rightarrow \C
\end{equation}
denote the stereographic projection.
We introduce the notation
\begin{equation}
\widetilde v_\infty \coloneqq v_\infty \circ St, \quad
\widetilde w\coloneqq w \circ St, \quad
\widetilde\omega_{\eta,\infty} \coloneqq [\widehat\omega_{\eta,\infty}(y)(1+\abs{y}^2)^2] \circ St.
\end{equation}
With a change of variables
\begin{equation}
\begin{aligned}
Q_{v_\infty}(w)
&= 2\int_\C \left(\widehat\omega_{\eta,\infty}^{-1}(- \Delta w  -S_{v_{\infty}}(\nabla v_{\infty})w \right) \cdot w \ \widehat\omega_{\eta,\infty} \ dvol_{\Sigma} \\
&= 2\int_{S^2} \left(\widetilde\omega_{\eta,\infty}^{-1}(- \Delta\widetilde w  -S_{\widetilde v_{\infty}}(\nabla\widetilde v_{\infty})\widetilde w  \right) \cdot\widetilde w \ \widetilde\omega_{\eta,\infty} \ dvol_{\Sigma} = Q_{\widetilde v_\infty}(\widetilde w),
\end{aligned}
\end{equation}
We introduce
\begin{equation}
\widetilde{\mathcal L}_{\eta,\infty}: V_{\widetilde v_{\infty}} \rightarrow V_{\widetilde v_{\infty}}; 
\qquad\widetilde{\mathcal L}_{\eta,\infty}(\widetilde w) \coloneqq2\ P_{\widetilde v_\infty} \left(\widetilde\omega_{\eta,\infty}^{-1}(- \Delta \widetilde w  -S_{\widetilde v_{\infty}}(\nabla\widetilde v_{\infty})\widetilde w) \right)
\end{equation}
Let
\begin{equation}
U_{u_\infty} = \left\{ w \in L^{2}_{\omega_{\eta,\infty}}(\Sigma; \R^{m}) \forwhich w(x) \in T_{u_\infty(x)} \mathcal N, \quad \text{for a.e. } x\in\Sigma \right\},
\end{equation}
and 
\begin{equation}
U_{\widetilde v_\infty} = \left\{ w \in L^{2}_{\widetilde\omega_{\eta,\infty}}(S^2; \R^{m}) \forwhich w(x) \in T_{ \widetilde v_\infty(x)} \mathcal N, \quad \text{for a.e. } x\in S^2 \right\}.
\end{equation}
In Lemma IV.5 of \cite{DGR22} the following result was shown:

\begin{lemma}
\begin{enumerate}
\item The separable Hilbert space $(U_{u_\infty}, \langle\cdot,\cdot\rangle_{\omega_{\eta,\infty}})$ has a Hilbert basis consisting of eigenfunctions of $\mathcal L_{\eta,\infty}$.
\item The separable Hilbert space $(U_{\widetilde v_\infty}, \langle\cdot,\cdot\rangle_{\widetilde \omega_{\eta,\infty}})$ has a Hilbert basis consisting of eigenfunctions of $\widetilde{\mathcal L}_{\eta,\infty}$.
\end{enumerate}
\end{lemma}

\noindent
We continue by introducing the limiting eigenspaces 
\begin{equation}
\mathcal E_{\eta,\infty}(\lambda) \coloneqq \{ w \in V_{u_\infty} \ ; \ \mathcal L_{\eta,\infty}(w)=\lambda w \},
\qquad
\widehat{\mathcal E}_{\eta,\infty}(\lambda) \coloneqq \{ w \in V_{v_\infty} \ ; \ \widehat{\mathcal L}_{\eta,\infty}(w)=\lambda w \}.
\end{equation}
And their nonpositive contribution
\begin{equation}
\mathcal E_{\eta,\infty}^0\coloneqq \bigoplus_{\lambda \le 0} \mathcal E_{\eta,\infty}(\lambda), \qquad
\widehat{\mathcal E}_{\eta,\infty}^0\coloneqq \bigoplus_{\lambda \le 0} \widehat{\mathcal E}_{\eta,\infty}(\lambda).
\end{equation}
In \cite{DGR22} in (IV.38) and (IV.45) the following result was shown:
\begin{lemma}
\label{LEMMA: Ind plus Null eq dim of eig for infty lim func}
\begin{equation}
\begin{aligned}
&i)\ \dim\left(\mathcal E_{\eta,\infty}^0\right) 
\le \operatorname{Ind}(u_\infty) + \operatorname{Null}(u_\infty), \\
&ii)\ \dim\left(\widehat{\mathcal E}_{\eta,\infty}^0\right) 
\le \operatorname{Ind}(\widetilde v_\infty) + \operatorname{Null}(\widetilde v_\infty),
\end{aligned}
\end{equation}
\end{lemma}

We consider the unit sphere (finite dimensional as the ambient space is finite dimensional) given by
\begin{equation}
\mathcal S_{\eta,k}^0 \coloneqq \left\{ w \in \bigoplus_{\lambda \le 0} \mathcal E_{\eta,k}(\lambda) \ ; \ \langle w,w \rangle_{\omega_{\eta,k}}=1 \right\}.
\end{equation}

\begin{lemma}
\label{LEMMA: On the non zero of the lim functs}
For any $k \in \N$ let $w_k \in \mathcal S_{\eta,k}^0$.
Then there exists a subsequence such that
\begin{equation}
w_k \rightharpoondown w_\infty, \text{ weakly in } W^{1,2}(\Sigma) \cap W^{2,2}_{loc}(\Sigma\setminus\{q\}),
\end{equation}
\begin{equation}
w_{k}(\delta_k y) \rightharpoondown \sigma_\infty(y), \text{ weakly in } W^{2,2}_{loc}(\C)
\end{equation}
and 
\begin{equation}
\text{ either }w_\infty \ne 0, \quad \text{ or } \sigma_\infty \ne 0.
\end{equation}
\end{lemma}

\begin{proof}
We have $Q_{u_k}(w_k)\le 0$.
With Lemma \ref{LEMMA: Linfty bd on na u in the necks unif} and Lemma \ref{LEMMA: prop on the seq mu eta k} we can estimate 
\begin{equation}
\begin{aligned}
&\hspace{-15mm}\int_\Sigma \left( 1+\abs{\nabla u_k}^2\right)^{p_k/2-1}  \mathbb I_{u_k}(\nabla u_k, \nabla u_k) \cdot \mathbb I_{u_k}(w_k,w_k) \ dvol_\Sigma \\
&\le C\norm{\left(1+\abs{\nabla u_k}^2 \right)^{\frac{p_k}{2}-1}}_{L^\infty(\Sigma)}  \norm{\frac{\abs{\nabla u_k}^2}{\omega_{\eta,k}}}_{L^{\infty}(\Sigma)} \int_{\Sigma} \abs{w_k}^2 \omega_{\eta,k} \ dvol_{\Sigma}\\ &\le C.
\end{aligned}
\end{equation}
This implies
\begin{equation}
\begin{aligned}
\int_{\Sigma} \abs{\nabla w_k}^2 \ dvol_{\Sigma}
&\le p_k \int_{\Sigma} \left( 1+\abs{\nabla u_k}^2\right)^{p_k/2-1} \abs{\nabla w_k}^2 \ dvol_{\Sigma} \\
& = \underbrace{Q_{u_k}(w_k) }_{\le 0}  \underbrace{-p_{k}(p_k-2) \int_{\Sigma } \left( 1+\abs{\nabla u_k}^2\right)^{p_k/2-2} \left(\nabla u_k \cdot \nabla w_k \right)^2  \ dvol_\Sigma}_{\le0} \\
&\hspace{20mm} + \underbrace{p_k \int_\Sigma \left( 1+\abs{\nabla u_k}^2\right)^{p_k/2-1} \mathbb I_{u_k}(\nabla u_k, \nabla u_k) \cdot \mathbb I_{u_k}(w_k,w_k) \ dvol_\Sigma}_{\le C} \\
&\le C.
\end{aligned}
\end{equation}
Therefore we may assume up to passing to subsequences that
\begin{equation}
\label{EQ: HINUJNFIUWEinwoienui9812}
w_k \rightharpoondown w_\infty \quad \text{ in } W^{1,2}(\Sigma) \qquad \text{ and } \qquad
\sigma_k(y)\coloneqq w_{k}(\delta_k y) \rightharpoondown \sigma_\infty(y) \quad \text{ in } W^{1,2}(\C).
\end{equation}
In the following we show
\\
\textbf{Claim 1:}
$\forall\eta>0:\exists C, k_0: \forall k\ge k_0:
\norm{\nabla^2 w_k}_{L^2(\Sigma\setminus B_{\eta})} \le C(\eta)$.
\\
\textbf{Proof of Claim 1:}
For $w\in V_{u_k}$ we consider the operator
\begin{equation}
\mathfrak E_{\eta,k}(w)^i \coloneqq- \partial_\alpha \left( A_{i,j}^{\alpha,\beta} \ \partial_\beta w^j \right),
\end{equation}
where
\begin{equation}
A_{i,j}^{\alpha,\beta} 
\coloneqq p_k (p_k-2)(1+\abs{\nabla u_k}^2)^{\frac{p_k}{2}-2}\ \partial_\alpha u_k^i\ \partial_\beta u_k^j
+ p_k (1+\abs{\nabla u_k}^2)^{\frac{p_k}{2}-1}\ \delta_{\alpha\beta}\ \delta_{ij}.
\end{equation}
There holds
\begin{equation}
\label{EQ: IJNIUqnufien12983JIN12dnunqdi918g}
\mathcal L_{\eta,k} (w)
= \omega_{\eta,k}^{-1}\ P_{u_k} \left[ \mathfrak E_{\eta,k}(w)-p_k(1+\abs{\nabla u_k}^2)^{\frac{p_k}{2}-1} S_{u_{k}}(\nabla u_{k}) w \right].
\end{equation}
Next, we show that the operator $\mathfrak E_{\eta,k}$ is elliptic in the sense that the coefficients satisfy for large $k$ the Legendre-Hadamard condition
\begin{equation}
\label{EQ: wbwef1715JNDWqoif9n3r13}
A_{i,j}^{\alpha,\beta} a_\alpha a_\beta b^i b^j
\ge c \abs{a}^2 \abs{b}^2, \qquad \forall a \in \R^2, \forall b \in \R^{m},
\end{equation}
as in section 3.4.1 in \cite{GM13}.
We can bound
\begin{equation}
\begin{aligned}
&\hspace{-20mm}\abs{p_k (p_k-2)(1+\abs{\nabla u_k}^2)^{\frac{p_k}{2}-2}\ \partial_\alpha u_k^i\ \partial_\beta u_k^j\ a_\alpha a_\beta b^i b^j} \\
&\le 2(n+1) \ p_k \underbrace{(p_k-2)}_{\rightarrow 0} \underbrace{\frac{\abs{\nabla u_k}^2}{1+\abs{\nabla u_k}^2}}_{\le 1} \underbrace{\norm{\left(1+\abs{\nabla u_k}^2 \right)^{\frac{p_k}{2}-1}}_{L^\infty(\Sigma)}}_{\le C} \abs{a}^2 \abs{b}^2 \\
&\le C (p_k-2) \abs{a}^2 \abs{b}^2,
\end{aligned}
\end{equation}
where we used also Lemma \ref{LEMMA: Linfty bd on na u in the necks unif}.
Hence, for large  $k$ we may assume that 
\begin{equation}
\abs{p_k (p_k-2)(1+\abs{\nabla u_k}^2)^{\frac{p_k}{2}-2}\ \partial_\alpha u_k^i\ \partial_\beta u_k^j\ a_\alpha a_\beta b^i b^j} \le \abs{a}^2 \abs{b}^2.
\end{equation}
This allows to bound
\begin{equation}
\begin{aligned}
A_{i,j}^{\alpha,\beta} a_\alpha a_\beta b^i b^j
\ge - \abs{a}^2 \abs{b}^2 + \underbrace{p_k(1+\abs{\nabla u_k}^2)^{\frac{p_k}{2}-1}}_{\ge 2}  \abs{a}^2 \abs{b}^2
\ge \abs{a}^2 \abs{b}^2.
\end{aligned}
\end{equation}
We have showed \eqref{EQ: wbwef1715JNDWqoif9n3r13} with constant $c=1$.
This proves that $\mathfrak E_{\eta,k}$ is an elliptic operator and the theory of elliptic systems as in section 4.3.1 of \cite{GM13} applies, i.e.
there exists some constant $C=C(\eta)>0$ which may depend on $\eta$ but not $k$ such that
\begin{equation}
\label{EQ: INUINiuwneg29e983rfn9}
\norm{\nabla^2 w_k}_{L^2(\Sigma \setminus B_\eta)}
\le C \left( \norm{w_k}_{W^{1,2}(\Sigma)} + \norm{\mathfrak E_{\eta,k} (w_k)}_{L^2(\Sigma \setminus B_{\frac{\eta}{2}})} \right).
\end{equation}
It remains to show that
\begin{equation}
\norm{\mathfrak E_{\eta,k} (w_k)}_{L^2(\Sigma \setminus B_{\frac{\eta}{2}})} \le C.
\end{equation}
To that end, we write with \eqref{EQ: IJNIUqnufien12983JIN12dnunqdi918g}
\begin{equation}
\begin{aligned}
\label{EQ: upINUFIPN9813rn1ior3n1io212e}
\mathfrak E_{\eta,k} (w_k) 
&= P_{u_k} \mathfrak E_{\eta,k} (w_k) + (id - P_{u_k}) \mathfrak E_{\eta,k} (w_k) \\
&= \omega_{\eta,k}\ \mathcal L_{\eta,k} (w_k) + P_{u_k} \left[p_k(1+\abs{\nabla u_k}^2)^{\frac{p_k}{2}-1} S_{u_{k}}(\nabla u_{k}) w_k \right] + (id - P_{u_k}) \mathfrak E_{\eta,k} (w_k).
\end{aligned}
\end{equation}
In the following we estimate the terms appearing in \eqref{EQ: upINUFIPN9813rn1ior3n1io212e} separately.
As by assumption $w_k \in \mathcal S_{\eta,k}^0$ we can write 
\begin{equation}
w_k = \sum_{j=1}^{N_k} c_k^j\ \phi_k^j,
\qquad \text{ where } \sum_{j=1}^{N_k} (c_k^j)^2=1
\end{equation}
and $\phi_k^1,\ldots,\phi_k^{N_k}$ is an orthonormal basis of $\oplus_{\lambda \le 0} \mathcal E_{\eta,k}(\lambda)$.
Then
\begin{equation}
\mathcal L_{\eta,k} (w_k) 
= \sum_{j=1}^{N_k} c_k^j \ \mathcal L_{\eta,k} (\phi_k^j)
= \sum_{j=1}^{N_k} c_k^j \ \lambda_{k}^j\ \phi_k^j.
\end{equation}
Hence,
\begin{equation}
\begin{aligned}
\label{EQ: IBNIUNFIPUNUIFWIN21}
\norm{\omega_{\eta,k}\ \mathcal L_{\eta,k} (w_k)}_{L^2(\Sigma \setminus B_{\frac{\eta}{2}})}
&\le \norm{\omega_{\eta,k}}_{L^\infty(\Sigma \setminus B_{\frac{\eta}{2}})}^{1/2} \norm{\mathcal L_{\eta,k} (w_k)}_{L^2_{\omega_{\eta,k}}(\Sigma\setminus B_{\frac{\eta}{2}})} \\
&\le C \norm{\mathcal L_{\eta,k} (w_k)}_{L^2_{\omega_{\eta,k}}(\Sigma)}
\le C \left(\sum_{j=1}^{N_k} (c_k^j \ \lambda_{k}^j)^2 \right)^{\frac{1}{2}}
\le C\ \inf \Lambda_{\eta,k}
\le C\  \mu_0,
\end{aligned}
\end{equation}
where we used Lemma \ref{LEMMA: prop on the seq mu eta k} and its notations as well as the fact that $\norm{\omega_{\eta,k}}_{L^\infty(\Sigma \setminus B_{\frac{\eta}{2}})}  \le C=C(\eta)$.
We also have
\begin{equation}
\begin{aligned}
\label{EQ: bfe238913bnip1nrBNIUdq}
&\norm{P_{u_k} \left[p_k(1+\abs{\nabla u_k}^2)^{\frac{p_k}{2}-1} S_{u_{k}}(\nabla u_{k}) w_k \right]}_{L^2(\Sigma \setminus B_{\frac{\eta}{2}})} 
\le C\norm{(1+\abs{\nabla u_k}^2)^{\frac{p_k}{2}-1} \abs{\nabla u_k}^2 w_k}_{L^2(\Sigma \setminus B_{\frac{\eta}{2}})} \\
&\hspace{30mm} \le C \norm{(1+\abs{\nabla u_k}^2)^{\frac{p_k}{2}-1} \abs{\nabla u_k}}_{L^\infty(\Sigma \setminus B_{\frac{\eta}{2}})} \mu_{\eta,k}^{\frac{1}{2}} \ \underbrace{\norm{w_k}_{L^2_{\omega_{\eta,k}}(\Sigma \setminus B_{\frac{\eta}{2}})}}_{\le1} \le C(\eta),
\end{aligned}
\end{equation}
where we used the strong convergence in \eqref{EQ: Cond of bubb conv ji0wr931jJGA} and also Lemma \ref{LEMMA: prop on the seq mu eta k} with its notations.
Now
\begin{equation}
\begin{aligned}
\label{EQ: inIFUPf293ngfgo4on2g}
(id - P_{u_k}) \mathfrak E_{\eta,k} (w_k)
&= (id - P_{u_k}) \Bigg[ p_k\ (p_k-2) \divergence \left( \left( 1+\abs{\nabla u_k}^2\right)^{p_k/2-2} \left(\nabla u_k \cdot \nabla w_k \right)\ \nabla u_k \right) \\
&\hspace{20mm} + p_k \divergence\left( \left( 1+\abs{\nabla u_k}^2\right)^{p_k/2-1} \nabla w_k \right)\Bigg] \\
&= p_k\ (p_k-2) \underbrace{(id - P_{u_k})\Bigg[ \nabla\left( \frac{\left(\nabla u_k \cdot \nabla w_k \right)}{1+\abs{\nabla u_k}^2} \right) \left( 1+\abs{\nabla u_k}^2\right)^{p_k/2-1} \cdot \nabla u_k \Bigg]}_{=0 \text{ (since }(id-P_{u_k})\partial_\alpha u =0)}  \\
&\hspace{10mm} -p_k\ (p_k-2) \frac{\left(\nabla u_k \cdot \nabla w_k \right)}{1+\abs{\nabla u_k}^2} \left( 1+\abs{\nabla u_k}^2\right)^{p_k/2-1} \mathbb{I}_{u_k}\big(\nabla u_k , \nabla u_k\big) \\
&\hspace{20mm} + p_k \ (id - P_{u_k})\Bigg[\divergence\left( \left( 1+\abs{\nabla u_k}^2\right)^{p_k/2-1} \nabla w_k \right)\Bigg],
\end{aligned}
\end{equation}
where we also used \eqref{EQ: Euler-Lagrange equations in div form for p harm}.
Recall that $P:\mathcal N \rightarrow \R^{m\times m}$ is the map that to any $q \in \mathcal N$ assigns the matrix corresponding to the orthogonal projection from $\R^m$ to $T_q\mathcal N$.
Using the facts
\begin{equation}
\nabla (P_{u_k}) = (DP)_{u_k}(\nabla u_k), \qquad
(id - P_{u_k})w_k =0, \qquad
\nabla (P_{u_k}) \cdot  w_k = (id-P_{u_k}) \nabla w_k,
\end{equation}
we get
\begin{equation}
\begin{aligned}
\label{EQ: njIFIN0913rj1f3okJNIqdw}
&(id - P_{u_k})\Bigg[\divergence\left( \left( 1+\abs{\nabla u_k}^2\right)^{p_k/2-1} \nabla w_k \right)\Bigg] \\
&= \divergence \left( \left( 1+\abs{\nabla u_k}^2\right)^{p_k/2-1} (id - P_{u_k})\nabla w_k  \right) + \left( 1+\abs{\nabla u_k}^2\right)^{p_k/2-1} \nabla (P_{u_k}) \cdot \nabla w_k \\
&= \divergence \left( \left( 1+\abs{\nabla u_k}^2\right)^{p_k/2-1} \nabla (P_{u_k}) \cdot  w_k \right) + \left( 1+\abs{\nabla u_k}^2\right)^{p_k/2-1} \nabla (P_{u_k}) \cdot \nabla w_k \\
&= \divergence \left( \left( 1+\abs{\nabla u_k}^2\right)^{p_k/2-1} \nabla (P_{u_k}) \right) \cdot w_k + 2\left( 1+\abs{\nabla u_k}^2\right)^{p_k/2-1} \nabla (P_{u_k}) \cdot \nabla w_k \\
&= \divergence \left( \left( 1+\abs{\nabla u_k}^2\right)^{p_k/2-1} (DP)_{u_k}(\nabla u_k) \right) \cdot w_k + 2\left( 1+\abs{\nabla u_k}^2\right)^{p_k/2-1} \nabla (P_{u_k}) \cdot \nabla w_k \\
&= - \left( 1+\abs{\nabla u_k}^2\right)^{p_k/2-1} \Big[(DP)_{u_k}\ \mathbb{I}_{u_k}\big(\nabla u_k , \nabla u_k\big)\Big] \cdot w_k \\
&\hspace{20mm}+ \left( 1+\abs{\nabla u_k}^2\right)^{p_k/2-1} \Bigg(\nabla\Big[(DP)_{u_k}\Big]\cdot (\nabla u_k)\Bigg)\cdot w_k \\
&\hspace{40mm} + 2\left( 1+\abs{\nabla u_k}^2\right)^{p_k/2-1} \Big[(DP)_{u_k}(\nabla u_k)\Big] \cdot \nabla w_k,
\end{aligned}
\end{equation}
where we also used \eqref{EQ: Euler-Lagrange equations in div form for p harm}.
We can with 
\begin{equation}
\abs{(DP)_{u_k}} \le \norm{DP}_{L^\infty}, \qquad
\abs{\nabla\Big[(DP)_{u_k}\Big]} \le \norm{D^2P}_{L^\infty} \abs{\nabla u_k},
\end{equation}
\eqref{EQ: inIFUPf293ngfgo4on2g} and \eqref{EQ: njIFIN0913rj1f3okJNIqdw} now bound
\begin{equation}
\begin{aligned}
\label{EQ: ZUBFIUEZBNUI1893rhn1r12e}
&\hspace{-5mm}\norm{(id - P_{u_k}) \mathfrak E_{\eta,k} (w_k)}_{L^2(\Sigma\setminus B_{\frac{\eta}{2}})} \\
&\le C \norm{(1+\abs{\nabla u_k}^2)^{\frac{p_k}{2}-1} \abs{\nabla u_k} }_{L^\infty(\Sigma \setminus B_{\frac{\eta}{2}})} \norm{\nabla w_k}_{L^2(\Sigma\setminus B_{\frac{\eta}{2}})} \\
&\hspace{20mm} + C \norm{(1+\abs{\nabla u_k}^2)^{\frac{p_k}{2}-1} \abs{\nabla u_k} }_{L^\infty(\Sigma \setminus B_{\frac{\eta}{2}})} \mu_{\eta,k}^{\frac{1}{2}} \norm{w_k}_{L^2_{\omega_{\eta,k}}(\Sigma\setminus B_{\frac{\eta}{2}})} \\
&\hspace{40mm} +C \norm{(1+\abs{\nabla u_k}^2)^{\frac{p_k}{2}-1} \abs{\nabla u_k} }_{L^\infty(\Sigma \setminus B_{\frac{\eta}{2}})} \norm{\nabla w_k}_{L^2(\Sigma\setminus B_{\frac{\eta}{2}})} \\
&\le C(\eta),
\end{aligned}
\end{equation}
where in the last line we used \eqref{EQ: HINUJNFIUWEinwoienui9812}, the strong convergence coming from \eqref{EQ: Cond of bubb conv ji0wr931jJGA} and Lemma \ref{LEMMA: prop on the seq mu eta k}.
Combining \eqref{EQ: INUINiuwneg29e983rfn9}, \eqref{EQ: upINUFIPN9813rn1ior3n1io212e}, \eqref{EQ: IBNIUNFIPUNUIFWIN21}, \eqref{EQ: bfe238913bnip1nrBNIUdq} and \eqref{EQ: ZUBFIUEZBNUI1893rhn1r12e} Claim 1 follows.
\\
Let now $\sigma_k(y) \coloneqq w_{k}(\delta_k y)$.
Proceeding similar as in the proof of Claim 1 we can also show
\\
\textbf{Claim 2:}
$\forall\eta>0:\exists C, k_0: \forall k\ge k_0:
\norm{\nabla^2 \sigma_k}_{L^2(B_{\frac{1}{\eta}})} \le C(\eta)$.
\\
With Claim 1 and Claim 2 we find that
\begin{equation}
\label{EQ: JINFU10930fi13fjm}
w_k \rightharpoondown w_\infty, \text{ weakly in } W^{2,2}_{loc}(\Sigma\setminus\{q\}),
\end{equation}
and
\begin{equation}
\label{EQ: UIN1902j1i0niqw1221mnq12D}
w_{k}(\delta_k y) \rightharpoondown \sigma_\infty(y), \text{ weakly in } W^{2,2}_{loc}(\C).
\end{equation}
It remains to show that either $w_\infty \ne 0$ or $\sigma_\infty \ne 0$.
For a contradiction assume that $w_\infty = 0$ and $\sigma_\infty = 0$.
Let $\chi \in C^\infty([0,\infty);[0,1])$ with $\chi=1$ on $[0,1]$ and $\chi=0$ on $[2,\infty)$.
Introduce the notation
\begin{equation}
\check w_k \coloneqq
w_k\ \chi\left(2 \frac{\left|x\right|}{\eta}\right)\left(1-\chi\left(\eta \frac{\left|x\right|}{\delta_k}\right)\right)  \in W^{1,2}_0(A(\eta,\delta_k);\R^{m})\cap V_{u_k}.
\end{equation}
Because of \eqref{EQ: JINFU10930fi13fjm}, \eqref{EQ: UIN1902j1i0niqw1221mnq12D} and because $w_\infty = 0$ and $\sigma_\infty = 0$ we find that
\begin{equation}
\label{EQ: BPNBFEIU193rnu13nijqUN129812}
\lim_{k \rightarrow \infty} \norm{\nabla(w_k- \check w_k)}_{L^2(\Sigma)}=0, \qquad
\lim_{k \rightarrow \infty} \norm{w_k- \check w_k}_{L^2(\Sigma)}=0
\end{equation}
We have
\begin{equation}
\begin{aligned}
\label{EQ: iwebfnwuUIN8913rn1d1312}
&\hspace{-10mm}\abs{Q_{u_k}(w_k)-Q_{u_k}(\check w_k)}\\
&\le p_k (p_k-2) \int_\Sigma \left( 1+\abs{\nabla u_k}^2\right)^{\frac{p_k}{2} -2} \abs{\left(\nabla u_k \cdot \nabla w_k \right)^2 - \left(\nabla u_k \cdot \nabla \check w_k \right)^2 } \ dvol_\Sigma \\
&\hspace{10mm}+ p_k \int_\Sigma \left( 1+\abs{\nabla u_k}^2\right)^{\frac{p_k}{2} - 1} \abs{\abs{\nabla w_k}^2-\abs{\nabla\check w_k}^2} \ dvol_\Sigma \\
&\hspace{20mm}+ p_k \int_\Sigma \left( 1+\abs{\nabla u_k}^2\right)^{\frac{p_k}{2} - 1} \abs{\nabla u_k}^2 \Big|\mathbb I_{u_k}(w_k,w_k) - \mathbb I_{u_k}(\check w_k,\check w_k)\Big| \ dvol_\Sigma \\
&\eqqcolon I + II + III
\end{aligned}
\end{equation}
First, with Lemma \ref{LEMMA: Linfty bd on na u in the necks unif} and \eqref{EQ: BPNBFEIU193rnu13nijqUN129812}
\begin{equation}
\begin{aligned}
I &\le p_k (p_k-2) \underbrace{\norm{\left(1+\abs{\nabla u_k}^2 \right)^{\frac{p_k}{2}-1}}_{L^\infty(\Sigma)}}_{\le C} \int_\Sigma \underbrace{\frac{\abs{\nabla u_k}^2}{1+\abs{\nabla u_k}^2}}_{\le 1} \left(\abs{\nabla w_k}^2 + \abs{\nabla\check w_k}^2 \right) \ dvol_\Sigma \\
&\le C (p_k-2) \int_\Sigma \left(\abs{\nabla w_k}^2 + \abs{\nabla\check w_k}^2 \right) \ dvol_\Sigma
\le C (p_k-2) \rightarrow 0, \qquad \text{ as } k \rightarrow \infty.
\end{aligned}
\end{equation}
Second, with Lemma \ref{LEMMA: Linfty bd on na u in the necks unif}, \eqref{EQ: JINFU10930fi13fjm}, \eqref{EQ: UIN1902j1i0niqw1221mnq12D}, \eqref{EQ: BPNBFEIU193rnu13nijqUN129812} and \eqref{EQ: Cond of bubb conv ji0wr931jJGA}
\begin{equation}
\begin{aligned}
II
&\le C \norm{\left(1+\abs{\nabla u_k}^2 \right)^{\frac{p_k}{2}-1}}_{L^\infty(\Sigma)}\int_\Sigma \abs{\abs{\nabla w_k}^2-\abs{\nabla\check w_k}^2} \ dvol_\Sigma \\
&\le C \underbrace{\int_{\Sigma\setminus B_{\frac{\eta}{2}}} \abs{\abs{\nabla w_k}^2-\abs{\nabla\check w_k}^2} \ dvol_\Sigma}_{\rightarrow 0, \text{ as }k \rightarrow \infty }
+C \underbrace{\int_{B_{\frac{2\delta_k}{\eta}}} \abs{\abs{\nabla w_k}^2-\abs{\nabla\check w_k}^2} \ dvol_\Sigma}_{\rightarrow 0, \text{ as }k \rightarrow \infty }.
\end{aligned}
\end{equation}
Recall now the orthonormal frame of the normal bundle introduced in \eqref{EQ: iwenf8932injIUJN128ee12}.\\
Third, with Lemma \ref{LEMMA: Linfty bd on na u in the necks unif}, \eqref{EQ: JINFU10930fi13fjm}, \eqref{EQ: UIN1902j1i0niqw1221mnq12D}, \eqref{EQ: BPNBFEIU193rnu13nijqUN129812}, \eqref{EQ: HINUJNFIUWEinwoienui9812} and \eqref{EQ: Cond of bubb conv ji0wr931jJGA}
\begin{equation}
\begin{aligned}
III
&\le C \norm{\left(1+\abs{\nabla u_k}^2 \right)^{\frac{p_k}{2}-1}}_{L^\infty(\Sigma)} \int_\Sigma \overbrace{\abs{\nabla u_k}^2}^{\le C(\eta)} \Big|\mathbb I_{u_k}(w_k,w_k) - \mathbb I_{u_k}(\check w_k,\check w_k)\Big| \ dvol_\Sigma \\
&\le C \int_{\Sigma} \Big|\mathbb I_{u_k}(w_k,w_k) - \mathbb I_{u_k}(\check w_k,\check w_k)\Big| \ dvol_\Sigma \\
&\le C \sum_{j=1}^{m-n} \int_{\Sigma} \Big| \langle D(\mathrm n_j)_{u_k} w_k,w_k \rangle - \langle D(\mathrm n_j)_{u_k} \check w_k,\check w_k \rangle \Big| \ dvol_\Sigma \\
&\le C \sum_{j=1}^{m-n} \int_{\Sigma} \Big| \langle D(\mathrm n_j)_{u_k} w_k,w_k-\check w_k \rangle\Big| + \Big| \langle D(\mathrm n_j)_{u_k} w_k-D(\mathrm n_j)_{u_k} \check w_k,\check w_k \rangle \Big| \ dvol_\Sigma \\
&\le C \sum_{j=1}^{m-n} \underbrace{\norm{w_k}_{L^2(\Sigma)}}_{\le C} \underbrace{\norm{w_k-\check w_k}_{L^2(\Sigma)}}_{\rightarrow 0, \text{ as }k \rightarrow \infty } + \underbrace{\norm{w_k-\check w_k}_{L^2(\Sigma)}}_{\rightarrow 0, \text{ as }k \rightarrow \infty } \underbrace{\norm{\check w_k}_{L^2(\Sigma)}}_{\le C}.
\end{aligned}
\end{equation}
Going back to \eqref{EQ: iwebfnwuUIN8913rn1d1312} we have shown $ \lim_{k \rightarrow\infty} \abs{Q_{u_k}(w_k)-Q_{u_k}(\check w_k)}=0.$
The fact that $Q_{u_k}(w_k) \le 0$ implies 
\begin{equation}
\label{EQ: vuicojnthbonurihy}
\limsup_{k\rightarrow\infty} Q_{u_k}(\check w_k) \le 0.
\end{equation}
But now also with  \eqref{EQ: JINFU10930fi13fjm} and \eqref{EQ: UIN1902j1i0niqw1221mnq12D}
\begin{equation}
\begin{aligned}
\abs{1- \int_{\Sigma} \abs{\check w_k}^2 \omega_{\eta,k}\ dvol_{\Sigma}}
&= \abs{\int_{\Sigma} \abs{ w_k}^2 \omega_{\eta,k}\ dvol_{\Sigma} - \int_{\Sigma} \abs{\check w_k}^2 \omega_{\eta,k}\ dvol_{\Sigma} } \\
& \le \underbrace{\int_{\Sigma\setminus B_{\frac{\eta}{2}}} \abs{\abs{ w_k}^2 - \abs{\check w_k}^2} \overbrace{\omega_{\eta,k}}^{\le C(\eta)} \ dvol_{\Sigma}}_{\rightarrow0, \text{ as }k \rightarrow\infty} + \underbrace{\int_{B_{\frac{2\delta_k}{\eta}}} \abs{\abs{ w_k}^2 - \abs{\check w_k}^2} \overbrace{\omega_{\eta,k}}^{\le C(\eta)} \ dvol_{\Sigma}}_{\rightarrow0, \text{ as }k \rightarrow\infty},
\end{aligned}
\end{equation}
which implies
\begin{equation}
\lim_{k\rightarrow\infty}
\int_{\Sigma} \abs{\check w_k}^2 \omega_{\eta,k}\ dvol_{\Sigma}=1.
\end{equation}
Since $\check w_k \in W^{1,2}_0(A(\eta,\delta_k);\R^{m})\cap V_{u_k}$ we have thanks to Theorem \ref{THM: JDNW892fejeujfwef32} for some constant $\overline\kappa >0$ 
\begin{equation}
\liminf_{k\rightarrow\infty} Q_{u_k} (\check w_k)
\ge \overline\kappa \ \lim_{k\rightarrow\infty} \int_{\Sigma} \abs{\check w_k}^2 \omega_{\eta,k}\ dvol_{\Sigma}
= \overline\kappa >0.
\end{equation}
This is a contradiction to \eqref{EQ: vuicojnthbonurihy} and we have shown that either $w_\infty \ne 0$ or $\sigma_\infty \ne 0$.
\end{proof}

We can finally show

\begin{proof}[Proof (of Theorem \ref{THM: usc morse index sacks-uhlenbeck with one bubble})]
By Lemma \ref{LEMMA: Ind plus Null eq dim of eigenspaces} and Lemma \ref{LEMMA: Ind plus Null eq dim of eig for infty lim func} it suffices to show that for $k \in \N$ large and $\eta>0$ small
\begin{equation}
\operatorname{dim}\left( \bigoplus_{\lambda\le 0} \mathcal E_{\eta,k}(\lambda) \right) 
\le \dim\left(\mathcal E_{\eta,\infty}^0\right) + \dim\left(\widehat{\mathcal E}_{\eta,\infty}^0\right).
\end{equation}
Let $N\in\N$ be fixed. 
For $k \in\N$ let $\phi_k^1,\ldots,\phi_k^N$ be a free orthonormal family of $U_{u_k}$ of eigenfunctions of the operator $\mathcal L_{\eta,k}$ with according negative eigenvalues $\lambda_{k}^1, \ldots, \lambda_{k}^N \le 0$.
For a contradiction we assume that 
\begin{equation}
\label{EQ: jnioecnHI1829he1n1ed231}
N > \dim\left(\mathcal E_{\eta,\infty}^0\right) + \dim\left(\widehat{\mathcal E}_{\eta,\infty}^0\right).
\end{equation}
By Lemma \ref{LEMMA: On the non zero of the lim functs} we find that up to subsequences
\begin{equation}
\label{EQ: iuwenfiUNIUND9183hrn9231}
\phi_k^j \rightharpoondown \phi_\infty^j, \text{ weakly in } W^{2,2}_{loc}(\Sigma\setminus\{q\})
\end{equation}
and 
\begin{equation}
\sigma_k^j(z) \coloneqq \phi_k^j(\delta_k^j z) \rightharpoondown \sigma_\infty^j(z), \text{ weakly in } W^{2,2}_{loc}(\C).
\end{equation}
Let $r>0$ and $w\in W^{1,2}(\Sigma;\R^m)$ with $\supp (w)\subset \Sigma \setminus B_r(q)$. Consider
\begin{equation}
\label{EQ: inwejnJININF130fn}
\begin{aligned}
\langle \mathcal L_{\eta,k} \phi_k^j , w \rangle_{\omega_{\eta,k}} 
&= \underbrace{p_k\ (p_k-2) \int_{\Sigma \setminus B_r(q)} \left( 1+\abs{\nabla u_k}^2\right)^{p_k/2-2} \left(\nabla u_k \cdot \nabla \phi_k^j \right) \left(\nabla u_k \cdot P_{u_k} \nabla w \right)  \ dvol_\Sigma}_{\eqqcolon I_{\eta,k}} \\
&\hspace{10mm}+\underbrace{p_k \int_{\Sigma \setminus B_r(q)} \left( 1+\abs{\nabla u_k}^2\right)^{p_k/2-1} \nabla  \phi_k^j \cdot P_{u_k} \nabla w \ dvol_\Sigma}_{\eqqcolon II_{\eta,k}} \\
&\hspace{20mm}-\underbrace{p_k \int_{\Sigma \setminus B_r(q)} \left( 1+\abs{\nabla u_k}^2\right)^{p_k/2-1}  S_{u_{k}}(\nabla u_{k}) \phi_k^j \cdot P_{u_k}w \ dvol_\Sigma}_{\eqqcolon III_{\eta,k}} \\
\end{aligned}
\end{equation}
First, with Lemma \ref{LEMMA: Linfty bd on na u in the necks unif}
\begin{equation}
\begin{aligned}
\abs{I_{\eta,k}} 
&\le p_k (p_k-2) \underbrace{\norm{\left(1+\abs{\nabla u_k}^2 \right)^{\frac{p_k}{2}-1}}_{L^\infty(\Sigma)}}_{\le C} \int_{\Sigma \setminus B_r(q)} \underbrace{\frac{\abs{\nabla u_k}^2}{1+\abs{\nabla u_k}^2}}_{\le 1} \abs{\nabla \phi_k^j} \ \abs{\nabla w}  \ dvol_\Sigma \\
&\le C (p_k-2) \underbrace{\norm{\nabla \phi_k^j}_{L^2(\Sigma \setminus B_r(q))}}_{\le C} \norm{\nabla w}_{L^2(\Sigma)}
\le C (p_k-2) \rightarrow0, \qquad \text{ as } k \rightarrow \infty.
\end{aligned}
\end{equation}
Second, using \eqref{EQ: iuwenfiUNIUND9183hrn9231} and Corollary \ref{COROLLARY: Limit of na uk to pk minus 2 eq 1} we know that
\begin{equation}
\left( 1+\abs{\nabla u_k}^2\right)^{p_k/2-1} \nabla  \phi_k^j \rightharpoondown \nabla  \phi_\infty^j, \text{ weakl in } W^{1,2}_{loc}(\Sigma\setminus\{q\})
\end{equation} 
and hence also with \eqref{EQ: Cond of bubb conv ji0wr931jJGA}
\begin{equation}
II_{\eta,k} \rightarrow 2 \int_\Sigma \nabla  \phi_\infty^j \cdot P_{u_\infty} \nabla w \ dvol_\Sigma,  \qquad \text{ as } k \rightarrow \infty.
\end{equation}
Third, using \eqref{EQ: iuwenfiUNIUND9183hrn9231}, \eqref{EQ: Cond of bubb conv ji0wr931jJGA} and Corollary \ref{COROLLARY: Limit of na uk to pk minus 2 eq 1} we know that
\begin{equation}
\left( 1+\abs{\nabla u_k}^2\right)^{p_k/2-1} S_{u_{k}}(\nabla u_{k}) \phi_k^j \rightharpoondown S_{u_{\infty}}(\nabla u_{\infty}) \phi_\infty^j, \text{ weakl in } W^{2,2}_{loc}(\Sigma\setminus\{q\})
\end{equation} 
and hence 
\begin{equation}
III_{\eta,k} \rightarrow 2 \int_\Sigma S_{u_{\infty}}(\nabla u_{\infty}) \phi_\infty^j \cdot P_{u_\infty} w \ dvol_\Sigma,  \qquad \text{ as } k \rightarrow \infty.
\end{equation}
Going back to \eqref{EQ: inwejnJININF130fn} we have shown that 
\begin{equation}
\langle \mathcal L_{\eta,k} \phi_k^j , w \rangle_{\omega_{\eta,k}} \rightarrow \langle \mathcal L_{\eta,\infty} \phi_\infty^j , w \rangle_{\omega_{\eta,\infty}} , \qquad \text{ as } k \rightarrow \infty.
\end{equation}
This means that
\begin{equation}
\mathcal L_{\eta,k} \phi_k^j \rightharpoondown  \mathcal L_{\eta,\infty} \phi_\infty^j, \text{ weakly in } W^{1,2}_{loc}(\Sigma\setminus\{q\}).
\end{equation}
This together with
\begin{equation}
\mathcal L_{\eta,k} \phi_k^j = \lambda_k^j \phi_k^j \rightharpoondown \lambda_\infty^j \phi_\infty^j, \text{ weakly in } W^{1,2}_{loc}(\Sigma\setminus\{q\})
\end{equation}
 gives
\begin{equation}
\mathcal L_{\eta,\infty} \phi_\infty^j = \lambda_\infty^j \phi_\infty^j \text{ in } \mathcal D^\prime(\Sigma\setminus\{q\}).
\end{equation}
Since $\phi_\infty^j\in W^{1,2}(\Sigma)$ we can deduce using the Lemma A.10 in \cite{DLRS25} on Sobolev capacity that indeed
\begin{equation}
\mathcal L_{\eta,\infty} \phi_\infty^j = \lambda_\infty^j \phi_\infty^j \text{ in } \Sigma .
\end{equation}
Similar one shows that
\begin{equation}
\widehat {\mathcal L}_{\eta,\infty} \sigma_\infty^j = \lambda_\infty^j \sigma_\infty^j \text{ in } \C.
\end{equation}
Now since by \eqref{EQ: jnioecnHI1829he1n1ed231} $N > \dim({\mathcal E}_{\eta,\infty}^0 \times \widehat{\mathcal E}_{\eta,\infty}^0)$ we have that the family $(\phi_\infty^j,\sigma_\infty^j)_{j=1\ldots N}$ is linearly dependent and we can find some $(c_\infty^1, \ldots, c_\infty^N)\neq 0$ such that
\begin{equation}
\label{EQ: wifnuJNIOFN103mnf223}
\sum_{j=1}^N c_\infty^j \phi_\infty^j=0 \qquad \text{ and } \qquad \sum_{j=1}^N c_\infty^j \sigma_\infty^j=0.
\end{equation}
Let 
\begin{equation}
w_k \coloneqq \frac{1}{\left(\sum_{j=1}^N (c_\infty^j)^2\right)^{\frac{1}{2}}} \sum_{j=1}^N c_\infty^j \phi_k^j.
\end{equation}
Then $w_k \in \mathcal S_{\eta,k}^0$ and by Lemma \ref{LEMMA: On the non zero of the lim functs} up to subsequences
\begin{equation}
w_k \rightharpoondown w_\infty, \text{ in } \dot W^{1,2}(\Sigma) \qquad \text{ and } \qquad
w_{k}(\delta_k y+x_k) \rightharpoondown \sigma_\infty(y), \text{ in } \dot W^{1,2}(\C)
\end{equation}
and either $w_\infty \ne 0$ or $\sigma_\infty \ne 0$.
But by \eqref{EQ: wifnuJNIOFN103mnf223} one has $(w_\infty,\sigma_\infty)=(0,0)$.
This is a contradiction. 
\end{proof}

\newpage
\newpage
\ \vspace{15mm}
\appendix
\section{Appendix}

For completeness, here we provide a proof of the lower semicontinuity of the Morse index in our setting of Sacks-Uhlenbeck sequences to a homogeneous manifold. 
In the following we are always working in the setting and with the notations  introduced in Section \ref{SECTION: Preliminary definition and results}.\\
(Recall for instance $u_k, u_\infty, v_k, v_\infty, V_u, Q_u(\cdot), \Sigma, \mathcal N$.)

\begin{proposition}[Lower Semicontinuity of Morse Index]\label{LSC}
For large $k$ there holds
\begin{equation}
\operatorname{Ind}_{E}(u_\infty) + \operatorname{Ind}_{E}(v_\infty)
\le \operatorname{Ind}_{E_p}(u_k)\,.
\end{equation}
\end{proposition}

\begin{proof}
We set  $N_1:= \operatorname{Ind}(u_\infty)$ and $N_2 := \operatorname{Ind}(v_\infty)$.
Let $w^1,\dots,w^{N_1}$ be a basis of 
\begin{equation}
\{ w \in V_{u_\infty} ; Q_{u_\infty}(w)<0 \}.
\end{equation}
and let $\sigma^1,\dots,\sigma^{N_2}$
\begin{equation}
\{ \sigma \in V_{v_\infty} ; Q_{v_\infty}(\sigma)<0 \}.
\end{equation}
 There holds
\begin{equation}
\label{EQ: oiuwnfer9NIOQn23dq}
(id-P_{u_\infty}) w^i=0,
\qquad \text{ and } \qquad
(id-P_{v_\infty}) \sigma^i=0,
~~~\mbox{for all}~ i.
\end{equation}
 
\textbf{1.} By Lemma A.10 in \cite{DLRS25} on Sobolev capacity there exists a sequence $(f_l^i)_l\subset W^{1,2}(\Sigma)$ and radii $r_l^i>0$ such that
\begin{equation}
\label{EQ: uinwfe09jij0i2d3nJNdq}
\lim_{l\rightarrow\infty}\norm{f_l^i- w^i}_{W^{1,2}(\Sigma)}=0, \qquad \forall i=1,\dots,N_1
\end{equation}
and with $\supp(f^i_l)\subset \Sigma \setminus B_{r_l^i}$.
For $l \in \N$, $k \in \N$ and $i=1,\dots,N_1$, let us introduce
\begin{equation}
w_{l,k}^i \coloneqq 
f^i_l- (id-P_{u_k}) f^i_l \qquad \text{in } \Sigma,
\end{equation}
where $P_q: \R^{m}\rightarrow T_{q} \mathcal N$ is the orthogonal projection for $q\in \mathcal N$.
One has $w_{l,k}^i \in V_{u_k}$.

{\bf Claim 1.} It holds:
\begin{equation}
\label{EQ: iuwnfenUINdiu23nf234f223ya}
\lim_{l\rightarrow\infty} \limsup_{k \rightarrow \infty} \norm{w^i_{l,k}-w^i}_{W^{1,2}(\Sigma)}=0.
\end{equation}
{\bf Proof of Claim 1.}
Let $\rho>0$.
Then for large $l\ge l_0(\rho)$, we have by \eqref{EQ: uinwfe09jij0i2d3nJNdq} that
\begin{equation}
\norm{f^i_{l}-w^i}_{W^{1,2}(\Sigma)} < \frac{\rho}{2 + 2\norm{\nabla u_\infty}_{L^\infty(\Sigma)}}
\end{equation}
For such a fixed $l\ge l_0$ we can bound
\begin{equation}
\begin{aligned}
&\hspace{-5mm}\norm{w^i_{l,k}-w^i}_{W^{1,2}(\Sigma)}\\
&\le \norm{f^i_{l}-w^i}_{W^{1,2}(\Sigma)}
+ \norm{(id-P_{u_k}) f^i_l}_{W^{1,2}(\Sigma \setminus B_{r^i_l})} \\
&\le \norm{f^i_{l}-w^i}_{W^{1,2}(\Sigma)}
+ \norm{(id-P_{u_k}) (f^i_l-w_i)}_{W^{1,2}(\Sigma \setminus B_{r^i_l})}
+ \norm{(id-P_{u_k}) w_i}_{W^{1,2}(\Sigma \setminus B_{r^i_l})} \\
&\le 2\norm{f^i_{l}-w^i}_{W^{1,2}(\Sigma)}
+ C\norm{\nabla u_k}_{L^\infty(\Sigma \setminus B_{r^i_l})}\norm{f^i_l-w_i}_{L^{2}(\Sigma)}
+ \norm{(id-P_{u_k}) w_i}_{W^{1,2}(\Sigma \setminus B_{r^i_l})} \\
&\le C\left(2 + \norm{\nabla u_k}_{L^\infty(\Sigma \setminus B_{r^i_l})}\right)\norm{f^i_{l}-w^i}_{W^{1,2}(\Sigma)}
+ \norm{(id-P_{u_k}) w_i}_{W^{1,2}(\Sigma \setminus B_{r^i_l})} \\
&\le C\frac{2 + \norm{\nabla u_k}_{L^\infty(\Sigma \setminus B_{r^i_l})}}{2 + 2\norm{\nabla u_\infty}_{L^\infty(\Sigma)}} \rho
+ \norm{(id-P_{u_k}) w_i}_{W^{1,2}(\Sigma \setminus B_{r^i_l})}
\end{aligned}
\end{equation}
By \eqref{EQ: Cond of bubb conv ji0wr931jJGA} we find that for $k\ge k_0(l)$
\begin{equation}
\frac{2 + \norm{\nabla u_k}_{L^\infty(\Sigma \setminus B_{r^i_l})}}{2 + 2\norm{\nabla u_\infty}_{L^\infty(\Sigma)}} 
\le \frac{2 + 2\norm{\nabla u_\infty}_{L^\infty(\Sigma \setminus B_{r^i_l})}}{2 + 2\norm{\nabla u_\infty}_{L^\infty(\Sigma)}} \le 1.
\end{equation}
Combining \eqref{EQ: oiuwnfer9NIOQn23dq} and \eqref{EQ: Cond of bubb conv ji0wr931jJGA} we get
\begin{equation}
\limsup_{k\rightarrow\infty}\norm{(id-P_{u_k}) w_i}_{W^{1,2}(\Sigma \setminus B_{r^i_l})}=0.
\end{equation}
This gives 
\begin{equation}
\lim_{l\rightarrow\infty} \limsup_{k \rightarrow \infty} \norm{w^i_{l,k}-w^i}_{W^{1,2}(\Sigma)}
< C\rho,
\end{equation}
which shows the claim 1.
We can now use \eqref{EQ: iuwnfenUINdiu23nf234f223ya} to get
\begin{equation}
\lim_{l\rightarrow\infty} \limsup_{k \rightarrow \infty}
\abs{Q_{u_k}(w^i_{l,k})-Q_{u_\infty}(w^i)}=0
\end{equation}
This implies that, for large $l$ and large $k$, we have 
\begin{equation}
\label{EQ: iunwef9u2u3n2193un12ejjNqd}
Q_{u_k}(w^i_{l,k}) <0.
\end{equation} \\
\textbf{2.} Let $(g_l^i)_{l}\subset W^{1,2}(\C)$ be a sequence and $R_l^i\nearrow\infty$ as $l\to +\infty$  be  such that   $\supp(g^i_l)\subset B_{R_l^i}$, and 
\begin{equation}
\label{EQ: uq9nfenUINIUDNui9qew}
\lim_{l\rightarrow\infty}\norm{g_l^i- \sigma^i}_{W^{1,2}(\C)}=0, \qquad \forall i=1,\dots,N_2.
\end{equation}
 For $l \in \N$, $k \in \N$ (with $\delta_k \le \frac{1}{R_l^i}$) and $i=1,\dots,N_2$, let us introduce
\begin{equation}
\sigma_{l,k}^i \coloneqq
\begin{cases}
g^i_l(\frac{\cdot}{\delta_k}) -(id-P_{u_k(\cdot)})g^i_l(\frac{\cdot}{\delta_k}) ,
 & \abs{x} \le \delta_k R^i_l \le 1 \\
0, &\text{else}.
\end{cases}
\end{equation}
One has $\sigma_{l,k}^i \in V_{u_k}$.\par

{\bf Claim 2.}  We have:
\begin{equation}
\label{EQ: ONJJD0i13diamndJN12q156}
\lim_{l\rightarrow\infty} \limsup_{k \rightarrow \infty} \norm{\sigma^i_{l,k}(\delta_k \cdot) -\sigma^i}_{W^{1,2}(\C)}=0.
\end{equation}
{\bf Proof of Claim 2.}
Let $\rho>0$.
Then for large $l\ge l_0(\rho)$, we have by \eqref{EQ: uq9nfenUINIUDNui9qew} that
\begin{equation}
\norm{g^i_{l}-\sigma^i}_{W^{1,2}(\C)} < \frac{\rho}{2 + 2\norm{\nabla v_\infty}_{L^\infty(\C)}}
\end{equation}
For such a fixed $l\ge l_0$ we can bound
\begin{equation}
\begin{aligned}
&\hspace{-5mm}\norm{\sigma^i_{l,k}(\delta_k \cdot) -\sigma^i}_{W^{1,2}(\C)}\\
&\le \norm{g^i_{l}-\sigma^i}_{W^{1,2}(\C)}
+ \norm{(id-P_{v_k}) g^i_l}_{W^{1,2}(B_{R^i_l})} \\
&\le \norm{g^i_{l}-\sigma^i}_{W^{1,2}(\C)}
+ \norm{(id-P_{v_k}) (g^i_l-\sigma^i)}_{W^{1,2}(B_{R^i_l})}
+ \norm{(id-P_{v_k}) \sigma_i}_{W^{1,2}(B_{R^i_l})} \\
&\le 2\norm{g^i_{l}-\sigma^i}_{W^{1,2}(\C)}
+ C\norm{\nabla v_k}_{L^\infty(B_{R^i_l})}\norm{g^i_l-\sigma_i}_{L^{2}(B_{R^i_l})}
+ \norm{(id-P_{v_k}) \sigma_i}_{W^{1,2}(B_{R^i_l})} \\
&\le C\left(2 + \norm{\nabla v_k}_{L^\infty(B_{R^i_l})}\right)\norm{g^i_{l}-\sigma^i}_{W^{1,2}(\C)}
+ \norm{(id-P_{v_k}) \sigma_i}_{W^{1,2}(B_{R^i_l})}\\
&\le C\frac{2 + \norm{\nabla v_k}_{L^\infty(B_{R^i_l})}}{2 + 2\norm{\nabla v_\infty}_{L^\infty(\C)}} \rho
+ \norm{(id-P_{v_k}) \sigma_i}_{W^{1,2}(B_{R^i_l})}
\end{aligned}
\end{equation}
By \eqref{EQ: Cond of bubb conv ji0wr931jJGA} we find that for $k\ge k_0(l)$
\begin{equation}
\frac{2 + \norm{\nabla v_k}_{L^\infty(B_{R^i_l})}}{2 + 2\norm{\nabla v_\infty}_{L^\infty(\C)}} 
\le \frac{2 + 2\norm{\nabla v_\infty}_{L^\infty(B_{R^i_l})}}{2 + 2\norm{\nabla v_\infty}_{L^\infty(\Sigma)}} \le 1.
\end{equation}
Combining \eqref{EQ: oiuwnfer9NIOQn23dq} and \eqref{EQ: Cond of bubb conv ji0wr931jJGA} we get
\begin{equation}
\limsup_{k\rightarrow\infty} \norm{(id-P_{v_k}) \sigma_i}_{W^{1,2}(B_{R^i_l})}=0.
\end{equation}
This gives 
\begin{equation}
\lim_{l\rightarrow\infty} \limsup_{k \rightarrow \infty} \norm{\sigma^i_{l,k}(\delta_k \cdot) -\sigma^i}_{W^{1,2}(\C)}
< C\rho,
\end{equation}
which shows the claim 2. \\
We can now use \eqref{EQ: ONJJD0i13diamndJN12q156} to get
\begin{equation}
\lim_{l\rightarrow\infty} \limsup_{k \rightarrow \infty}
\abs{Q_{v_k}(\sigma^i_{l,k})-Q_{v_\infty}(\sigma^i)}=0.
\end{equation}
This implies that for large $l$ and large $k$ we have
\begin{equation}
\label{EQ: jiNIJDWIQIIOEM893n1e9n}
Q_{v_k}(\sigma^i_{l,k}) <0.
\end{equation} \\
\textbf{3.}
Now we claim that for large $l$ and large $k$ the family
\begin{equation}
\mathcal B_{l,k} \coloneqq \{ w_{l,k}^1, \dots, w_{l,k}^{N_1}, \sigma_{l,k}^1, \dots, \sigma_{l,k}^{N_2} \} \subset V_{u_k}
\end{equation}
is linearly independent.
(In the following $G(B)$ denotes the determinant of the Gram matrix of a given basis $B$.)
As $(w^i)_{i=1,\dots, N_1}$ is a linear independent family we know that the determinant of the Gram matrix is non-zero, i.e. there is some $\kappa_1>0$ such that
\begin{equation}
\label{EQ: iu2n3fpfdwi0ej2013frn381209HZu1}
G(\{w_1,\dots,w_{N_1}\})= \det \Big[ \left( \langle w^i , w^j \rangle_{L^2(\Sigma)} \right)_{i,j} \Big]\ge \kappa_1 >0.
\end{equation}
Similar as $(\sigma^i)_{i=1,\dots, N_1}$ is a linear independent family we know that the determinant of the Gram matrix is non-zero, i.e. there is some $\kappa_2>0$ such that
\begin{equation}
\label{EQ: inHUIniuqefni23f9013r}
G(\{\sigma_1,\dots,\sigma_{N_2}\})=
\det \Big[ \left( \langle \sigma^i , \sigma^j \rangle_{L^2(\C)} \right)_{i,j} \Big]\ge \kappa_2 >0.
\end{equation}
Now note that as $\supp(w^i_{l,k})\subset\Sigma\setminus B_{r^i_l}$ and $\supp(\sigma^i_{l,k})\subset B_{R^i_l\delta_k}$ for large $l$ and large $k$
we will find that
\begin{equation}
\langle w^i_{l,k} , \sigma^j_{l,k} \rangle_{L^2(\Sigma)} = 0, \qquad \forall i=1,\dots,N_1,\forall j=1,\dots,N_2.
\end{equation}
Hence, if we compute the Gram matrix of $\mathcal B_{l,k}$ we have
\begin{equation}
\label{EQ: inuweunfiNIJ3209r1nirqfuz1ws0jw}
G(\mathcal B_{l,k})=
\det \Bigg[ \left( \langle w^i_{l,k} , w^j_{l,k} \rangle_{L^2(\Sigma)} \right)_{i,j} \Bigg]\
\det \Bigg[ \left( \langle \sigma^i_{l,k} , \sigma^j_{l,k} \rangle_{L^2(\Sigma)}  \right)_{i,j} \Bigg]
\end{equation}
By \eqref{EQ: iuwnfenUINdiu23nf234f223ya} and \eqref{EQ: ONJJD0i13diamndJN12q156} we know that 
\begin{equation}
\label{EQ: ijnwfeuinIUui123d1}
 \langle w^i_{l,k} , w^j_{l,k} \rangle_{L^2(\Sigma)} \rightarrow \langle w^i , w^j \rangle_{L^2(\Sigma)},
\qquad
\langle \sigma^i_{l,k} , \sigma^j_{l,k} \rangle_{L^2(\Sigma)} \rightarrow \langle \sigma^i , \sigma^j\rangle_{L^2(\C)},
\end{equation}
as $k \rightarrow \infty$ and $l \rightarrow \infty$.
Combining \eqref{EQ: ijnwfeuinIUui123d1} with \eqref{EQ: inuweunfiNIJ3209r1nirqfuz1ws0jw} and 	\eqref{EQ: iu2n3fpfdwi0ej2013frn381209HZu1}, \eqref{EQ: inHUIniuqefni23f9013r} we find  for large $l$ and large $k$ that
\begin{equation}
G(\mathcal B_{l,k}) \ge \frac{\kappa_1 \kappa_2}{2}>0.
\end{equation}
As the determinant of the Gram matrix of $\mathcal B_{l,k}$ is non-zero we deduce that the family $\mathcal B_{l,k}$ is linearly independent.
This with \eqref{EQ: iunwef9u2u3n2193un12ejjNqd} and \eqref{EQ: jiNIJDWIQIIOEM893n1e9n} gives
\begin{equation}
N_1+N_2= \dim(span(\mathcal B_{l,k})) \le \dim(\{ w \in V_{u_k} ; Q_{u_k}(w)<0 \}) = \operatorname{Ind}(u_k).
\end{equation}
This concludes the proof of Proposition \ref{LSC}.
\end{proof}


\end{document}